\documentclass[12pt,english]{article}
\usepackage[T1]{fontenc}
\usepackage[latin9]{inputenc}
\usepackage[margin=1.23in]{geometry}
\usepackage{color}
\usepackage{amsmath}
\usepackage{amsthm}
\usepackage{amssymb}
\PassOptionsToPackage{normalem}{ulem}
\usepackage{ulem}
\usepackage{enumerate}

\makeatletter 
\newcommand{\cov}{{\operator@font cov}}
\newcommand{\var}{{\operator@font var}}
\newcommand{\vol}{{\operator@font vol}}
\newcommand{\corr}{{\operator@font corr}}
\newcommand{\diam}{{\operator@font diam}}
\newcommand{\Av}{{\operator@font Av}}
\newcommand{\trig}{{\operator@font trig}}
\newcommand{\Enh}{{\operator@font Enh}}
\newcommand{\EEnh}{\overline {\operator@font Enh}}
\newcommand{\Com}{{\operator@font Com}}
\newcommand{\Clus}{{\operator@font Clus}}
\makeatother

\makeatletter
\numberwithin{equation}{section}
\theoremstyle{plain}
\newtheorem{thm}{\protect\theoremname}[section]
\theoremstyle{remark}
\newtheorem{rem}{\protect\remarkname}[section]
\theoremstyle{plain}
\newtheorem{lem}{\protect\lemmaname}[section]
\ifx\proof\undefined\
  \newenvironment{proof}[1][\proofname]{\par
    \normalfont\topsep6\p@\@plus6\p@\relax
    \trivlist
    \itemindent\parindent
    \item[\hskip\labelsep
          \scshape
      #1]\ignorespaces
  }{%
    \endtrivlist\@endpefalse
  }
  \providecommand{\proofname}{Proof}
\fi
\theoremstyle{assumption}
\newtheorem{ass}{\protect\assumptionname}[section]
\theoremstyle{definition}
\newtheorem{defn}{\protect\definitionname}[section]
\theoremstyle{plain}

\theoremstyle{definition}

\theoremstyle{plain}
\newtheorem{prop}{\protect\propositionname}[section]
\theoremstyle{definition}
\newtheorem{rmk}{\protect\remarkname}[section]

\usepackage{comment}
\usepackage{graphicx}
\usepackage{float}
\usepackage{hyperref}
\usepackage{xcolor}

\newcommand{\E}{\mathbf{E}}
\renewcommand{\P}{\mathbf{P}}

\newcommand{\aA}{\mathcal A}

\newcommand{\eE}{\mathcal E}
\newcommand{\fF}{\mathcal F}
\newcommand{\gG}{\mathcal G}
\newcommand{\iI}{\mathcal I}

\newcommand{\mM}{\mathcal M}
\newcommand{\nN}{\mathcal N}
\newcommand{\oO}{\mathcal O}

\newcommand{\rR}{\mathcal R}
\newcommand{\sS}{\mathcal S}
\newcommand{\tT}{\mathcal T}

\newcommand{\EE}{\mathbb{E}}
\newcommand{\HH}{\mathbb{H}}
\newcommand{\NN}{\mathbb{N}}
\newcommand{\PP}{\mathbb{P}}

\newcommand{\RR}{\mathbb{R}}

\newcommand{\1}{\mathbf{1}}

\newcommand{\fC}{\mathfrak{C}}
\newcommand{\fS}{\mathfrak{S}}

\newcommand{\eps}{\varepsilon}

\hypersetup{
     colorlinks   = true,
     citecolor    = blue,
     linkcolor    = blue
}

\usepackage[title]{appendix}

\date{}

\makeatother

\usepackage{babel}
\providecommand{\assumptionname}{Assumption}
\providecommand{\corollaryname}{Corollary}
\providecommand{\definitionname}{Definition}
\providecommand{\examplename}{Example}
\providecommand{\lemmaname}{Lemma}
\providecommand{\propositionname}{Proposition}
\providecommand{\remarkname}{Remark}
\providecommand{\theoremname}{Theorem}

\begin{document}

\title{Parabolic Anderson Model in the Hyperbolic Space. Part II: Quenched Asymptotics}

\author{Xi Geng\thanks{School of Mathematics and Statistics, University of Melbourne, Australia.
Email: xi.geng@unimelb.edu.au.},$\ $ Sheng Wang\thanks{School of Mathematics and Statistics, University of Melbourne, Australia.
Email: sheng.wang5@student.unimelb.edu.au.}$\ $ and Weijun Xu\thanks{Institute for Theoretical Sciences, Westlake University, China. Email: xuweijun@westlake.edu.cn. } }

\maketitle

\begin{abstract}
We establish the exact quenched asymptotic
growth of the solution to the parabolic Anderson model (PAM) in the hyperbolic space with a regular, stationary, time-independent Gaussian potential. More precisely, we show that there exists a deterministic constant $L^*>0$ such that with probability one, the solution $u$ to PAM with constant initial data has pointwise growth asymptotics
\begin{equation*}
    u(t,x)\sim e^{L^{*}t^{5/3}+o(t^{5/3})}
\end{equation*}
as $t \rightarrow +\infty$. Both the power $t^{5/3}$ and the exact value of $L^*$ are different from their counterparts in the Euclidean situation. They are determined through an explicit optimisation procedure. Our proof relies on certain fine localisation techniques, which also reveals a stronger non-Euclidean localisation mechanism. 
\end{abstract}

\setcounter{tocdepth}{2}
\tableofcontents

\section{Introduction}

\subsection{The problem and main result}

This article is the second part of our investigation on the long time asymptotics of the parabolic Anderson model (PAM) on the hyperbolic space with a regular, stationary, time-independent Gaussian potential. We consider the solution $u$ to
\begin{equation}\label{eq:PAMIntro}
\partial_{t}u=\Delta u+\xi\cdot u, \qquad u(0,\cdot) \equiv 1
\end{equation}
on the standard hyperbolic space $\HH^d$ with constant curvature $\kappa \equiv -1$. Here $\xi$ is a Gaussian field satisfying the following assumption. 

\begin{ass} \label{ass:xi}
    $\xi$ is a real-valued, centered, time-independent, stationary Gaussian field on $\HH^d$ defined on some probability space $(\Omega,\mathcal{F},\mathbf{P})$. Furthermore, its covariance function\footnote{Stationarity means for every $n$, every $x_1, \dots, x_n \in \HH^d$ and every $g \in SO^{+}(d,1)$ (the group of orientation preserving isometries over $\HH^d$), we have
    \begin{equation*}
        \big( \xi(g \cdot x_1), \dots, \xi (g \cdot x_n) \big) \stackrel{\text{law}}{=} \big( \xi(x_1), \dots, \xi(x_n) \big)\;.
    \end{equation*}By stationarity, the covariance function must be a function of distance only (\cite[Lemma~2.1]{GX25}).}
    \begin{equation*}
        C (x,y) = C \big( d(x,y) \big) = \mathbf{E} \big[\xi(x) \xi(y) \big]
    \end{equation*}
    has two continuous derivatives on $\RR^+$, and there exists $r_0>0$ such that $C (r) = 0$ whenever $r > r_0$. Here $d$ denotes the hyperbolic distance. 
\end{ass}

Except for the finiteness of correlation length,
the setting here is identical to the one in the article \cite{GX25}. A simple way to construct such a $\xi$ is to convolute the white noise on the isometry group with a smooth and compactly supported mollifier. 

In the first part \cite{GX25}, we established the second order annealed asymptotics and demonstrated the same moment intermittency property as in the Euclidean case. One interesting phenomenon there is that the fluctuation exponent is determined by the geometry
of the \textit{Euclidean} (rather than hyperbolic) Laplacian. The purpose of the present article is to investigate the quenched asymptotic behaviour of the PAM. Our main theorem
is stated as follows. 

\begin{thm} \label{thm:MainThm}
Let $u$ be the solution to \eqref{eq:PAMIntro}. There exists a deterministic constant $L^*>0$ such that for every $x \in \HH^d$ and almost every realisation of $\xi$, we have
\begin{equation} \label{eq:MainThm}
    \lim_{t\rightarrow\infty} \frac{1}{t^{5/3}} \log u(t,x) = L^*\;.
\end{equation}
The value of $L^*$ is determined through the optimisation problem
\begin{equation} \label{eq:OptimalExp}
    L^* = \sup_{\theta \in [0,1], \, K>0} \; \Big\{ (1-\theta) \sqrt{2 (d-1) \sigma^2 K} - \frac{K^2}{4 \theta} \Big\}\;,
\end{equation}
where $\sigma^{2} \triangleq{\rm Var}[\xi(y)]$ is the variance of $\xi$. 
\end{thm}

\begin{rem} \label{rem:DelIC}
The solution $u(t,x)$ to \eqref{eq:PAMIntro} with constant initial condition $1$ is the same as the total mass (whole-space integral) of the solution to the same equation but with $\delta_x$-initial condition. Furthermore, by stationarity, $\big(u(t,x)\big)_{t \geqslant 0}$ has the same law for every $x \in \HH^d$. Hence, it suffices to consider $x=o$, the base point of $\HH^d$. 

The optimisation problem \eqref{eq:OptimalExp} can also be solved explicitly. Indeed, the optimal value $L^*$ and the pair of optimisers $(\theta^*, K^*)$ are given by
\begin{equation} \label{eq:optimal_values}
    L^* = \frac{3\cdot2^{4/3}}{5^{5/3}} \big( (d-1) \sigma^{2} \big)^{2/3}\;, \qquad \theta^* = \frac{1}{5}\;, \quad K^* = \frac{2^{5/3}}{5^{4/3}} \cdot \big( (d-1) \sigma^2 \big)^{1/3}\;.
\end{equation}
\end{rem}
Unlike the moment asymptotics, the almost sure behaviour revealed
by Theorem \ref{thm:MainThm} is very different from the Euclidean
situation. As its proof reveals, the PAM also undergoes a
stronger non-Euclidean localisation mechanism through its Feynman-Kac representation.

\subsection{Background and motivation}
\label{subsec:Background}

The study of the asymptotic behaviour of directed polymers in random environments have attracted much interest over the past decades in both discrete and continuous settings (cf. \cite{Com17, K\"on16} and the references therein). Mathematically, a random polymer is described by a random walk or a Brownian motion propagating through space where random rewards and penalties are distributed. One important question is to understand how the global behaviour of the polymer (the random walk or Brownian motion) is affected by the impurities of the random environment. A typical phenomenon is that the polymer would exhibit certain localisation property (e.g. favouring particular locations or trajectories) under the polymer measure (Gibbs measure). Such behaviour is intimately related to the (quenched) long time asymptotics of the so-called partition function, which is the solution to the parabolic Anderson model with potential function describing the random environment. 

The long time behaviour of PAM in Euclidean settings (over $\mathbb{Z}^{d}$ or $\mathbb{R}^{d}$ with various potentials) have been extensively
studied in the literature. The monograph \cite{K\"on16} contains an excellent exposition on various asymptotic results in the discrete setting. The annealed and quenched asymptotics for the continuous PAM with a regular Gaussian potential were studied in the seminal works \cite{CM95, GK00, GKM00}. The quenched asymptotics was extended to singular Gaussian potentials in \cite{Che14} and the white noise potential (in $d=2$) in \cite{KPvZ22}.

The main motivation for our work is to investigate a similar model in a different geometric setting and look for possible new non-Euclidean asymptotic behaviour and localisation mechanism. As an initial attempt, we work with the standard hyperbolic space where the global geometry and large scale behaviour of Brownian motion differ from the Euclidean situation in several fundamental ways. The underlying random environment is described by a time-independent, regular Gaussian
potential whose distribution is invariant under isometries. In the first part of our investigation \cite{GX25}, we established the annealed asymptotics for the solution to \eqref{eq:PAMIntro}. It turns out that the moment asymptotics is identical to the Euclidean result established in \cite{GK00} and the PAM exhibits the same moment intermittency property as in the Euclidean case. An interesting point there is that the fluctuation exponent is described by a variational problem induced by the \textit{Euclidean} (rather than hyperbolic) Laplacian. 

\medskip
\textit{Main novelty}. The goal of the present article
is to investigate the quenched asymptotics of the solution to \eqref{eq:PAMIntro}. Our main finding is that the almost sure growth of the solution, as revealed by (\ref{eq:MainThm}), is much faster
than the Euclidean situation where the growth rate was found to be (cf. \cite{CM95, GKM00})\begin{equation}
u(t,o)=\exp \big(\sqrt{2d \sigma^{2}} \, t\sqrt{\log t} + o(t\sqrt{\log t}) \big) \quad \text{as }t\rightarrow\infty.\label{eq:Euclidean}
\end{equation} The $t^{5/3}$-rate and the exact growth constant in \eqref{eq:MainThm} and \eqref{eq:OptimalExp} are determined through an explicit optimisation procedure from the competition between the reward from peak values of the Gaussian field and the cost of traveling to those regions. There is a new and interesting localisation mechanism for the Brownian motion in this situation: the Brownian motion over $[0,t]$ intends to reach the peak region of the Gaussian field \textit{over a definite ball} (precisely, with radius $K^{*}t^{4/3}$ where $K^{*}$ is an explicit number) \textit{at a definite proportion of time} (precisely, at $s=t/5$) and stay there for the rest of time. This localisation mechanism is different from the Euclidean situation (cf. Section \ref{sec:ComEuc} below for a more detailed comparison). We will discuss our methodology for proving this result as well as the main difficulties in Section~\ref{subsec:Heu} below. 

\medskip

\textit{Related works in non-Euclidean setting.} There are several related works on the behaviour of PAM in non-Euclidean spaces. In a closely related setting (discrete analogue of hyperbolic space), \cite{HKS21} established the quenched asymptotics on a Galton-Watson tree and \cite{HW23} established the annealed asymptotics on a regular tree. In both works, the random potential is assumed to have a double-exponential tail. Hence, the asymptotics theorems obtained there are quite different from Theorem~\ref{thm:MainThm} and the annealed behaviour presented in \cite{GX25}. 

In continuous setting, PAM with white noise potential and the analysis of the singular Anderson operator on Riemann surfaces (heat kernel estimates, spectral properties etc.) were studied in depth in a series of works \cite{DDD18, Mou22, BDM25}. The general framework for singular SPDEs in manifolds was established recently by \cite{HS23}. 

\cite{BCO24} considered PAM with a time-dependent (white in time and coloured in space) Gaussian potential on a Cartan-Hadamard manfold, where the authors obtained well-posedness as well as moment estimates for the solution. The role of non-positive curvature (and global geometry) also plays an essential role in the recent work \cite{CO25} where similar questions as in \cite{BCO24} were investigated. Other works related to the asymptotic behaviour of PAM with a time-dependent Gaussian potential in a geometric setting include e.g. \cite{TV02, BOT23, BCH24, COV24}. We should point out that the case of time-dependent potential is different from the time-independent case in several fundamental ways (the asymptotic results, methodology and underlying mechanism have very different nature in these two settings), and is in general much more challenging. The current work focus on the time-independent case.

\subsection{Outline of the proof}
\label{subsec:Heu}

In this section, we discuss how Theorem \ref{thm:MainThm} is proved
at a heuristic but rather non-rigorous level. We will also describe
the localisation mechanism for achieving the optimal growth rate $e^{L^{*}t^{5/3}}$ ($L^{*}$ is the exact constant given in (\ref{eq:MainThm})) and
compare such mechanism to the Euclidean situation. 

As mentioned in Remark~\ref{rem:DelIC}, by stationarity, we only need to consider $x=o$, the base point of $\HH^d$. The proof of Theorem~\ref{thm:MainThm} consists of two parts: 
\begin{equation} \label{eq:intro_low_upper}
    \underset{t\rightarrow\infty}{\underline{\lim}}\frac{1}{t^{5/3}}\log u(t,o)\geqslant L^{*} \quad \text{and} \quad \underset{t\rightarrow\infty}{\overline{\lim}}\frac{1}{t^{5/3}}\log u(t,o)\leqslant L^{*}\;, \quad \P-\text{a.s.}\;.
\end{equation}
The starting point is the Feynman-Kac formula to represent the solution $u$ to \eqref{eq:PAMIntro} as
\begin{equation} \label{eq:FKIntro}
u(t,o) = \EE_o \big[ e^{\int_{0}^{t}\xi(W_{s})ds} \big]\;,
\end{equation}
where $\EE_o$ is the expectation of the hyperbolic Brownian motion $W$ starting from $o$. 

We also need the sample functions of the Gaussian field $\xi$ admit the following almost-sure growth property:
\begin{equation}
\max_{x \in B(o,R)} \xi(x) \approx \sqrt{2(d-1) \sigma^{2} R}\ \ \ \text{as }R\rightarrow\infty.\label{eq:XiGrowthIntro}
\end{equation}
According to (\ref{eq:FKIntro}) and (\ref{eq:XiGrowthIntro}), it
is natural to expect that the growth of $u(t,o)$ is intimately related
to the quantitative competition between peak values of $\xi$ (reward)
and probabilities of reaching and staying at those peak regions by
the Brownian motion $W$ (cost). This viewpoint is robust for studying this
kind of models in various general settings (Euclidean / discrete /
other potentials).

\subsubsection{Heuristics on the optimal Brownian scenario and growth exponent}
\label{subsec:StratLow}

To prove the lower asymptotics in \eqref{eq:intro_low_upper}, we need to identify a particular Brownian scenario under which the growth rate $e^{L^{*}t^{5/3}}$ is attained. This Brownian scenario is expected to be the ``optimal'' one, and also illustrates the reason of the $t^{5/3}$ rate on the exponential as well as the value of the constant $L^*$.

Let $K(t)$ and $\theta(t)$ be parameters to be determined later
on ($K(t)\rightarrow\infty$ as $t\rightarrow\infty$ and $\theta(t)\in[0,t]$). On the ball $B(o,K(t))$, the Gaussian field $\xi$ achieves and maintain its maximal value \eqref{eq:XiGrowthIntro} (with $R = K(t)$) in a small neighbourhood around a (random) point $z^t$ near the boundary $\partial B(o,K(t))$. We consider the following scenario for the Brownian motion which potentially yields the main contribution to $u(t,o)$: 
\begin{equation} \label{eq:BM_scenario}
    \begin{split}
    W \; &\text{enters a small neighbourhood of} \; z^t \; \text{at time around} \; \theta(t)\\
    &\text{and stays there over the remaining time} \; [\theta(t), t]\;.
    \end{split}
\end{equation}
By considering the above scenario, and assuming for the moment that both the contribution of $\xi$ from the time $[0, \theta(t)]$ and the staying probability inside a small island in the remaining time $[\theta(t), t]$ are negligible, one can formally write down a lower bound
\begin{equation} \label{eq:FormalLow}
    \begin{split}
    u(t,o) &\gtrsim \big( 1 + o_t(1) \big)  \, e^{(t-\theta(t)) \cdot \sqrt{2(d-1) \sigma^{2} K(t)}} \times \PP_o \big( W_{\theta(t)} \approx z^t \big)\\
    &\approx \exp \bigg( \big( t - \theta(t) \big) \cdot \sqrt{2 \sigma^{2} (d-1) K(t)} \, - \, \frac{K^2(t)}{4 \theta(t)} \bigg)\;.
    \end{split}
\end{equation}
The exponential term in the first line comes from \eqref{eq:XiGrowthIntro} with $R = K(t)$, and the ``$\approx$'' in the second line follows from that the location $z^t$ is near the boundary of $\partial B(o, K(t))$ and the heat kernel estimate for hyperbolic Brownian motion to estimate $\PP_o \big( W_{\theta t} \approx z^t \big)$. 

By elementary calculations, one can see that the quantity inside the exponential of the right hand side above is positive and achieves the largest power in $t$ when $\theta(t)$ is proportional to $t$ and $K(t)$ is proportional to $t^{4/3}$. Therefore, one is led to taking 
\begin{equation*}
    \theta(t) = \theta t\;, \quad K(t) = K t^{4/3}
\end{equation*}
for $\theta \in [0,1]$ and $K>0$. With these choices, the quantity on the exponential in \eqref{eq:FormalLow} is precisely the right hand side of \eqref{eq:OptimalExp} for the pair $(\theta,K)$ (multiplied by $t^{5/3}$). One then optimises over $\theta \in [0,1]$ and $K>0$ to obtain the following potentially optimal Brownian scenario: 
\begin{equation} \label{eq:BM_scenario_optimal}
    \begin{split}
    \eE_t^* = \Big\{W \; &\text{travels almost like a geodesic to} \; z^t \; (\text{distance} \; K^* t^{4/3} \; \text{away from} \; o)\\
    &\text{in time} \; [0, \theta^* t]\,, \; \text{and stays there over the remaining time} \; [\theta^* t, t] \Big\}\;,
    \end{split}
\end{equation}
where the pair $(\theta^*, K^*)$ is given in \eqref{eq:optimal_values}. Under this scenario, one expects the lower bound
\begin{equation*}
    u(t,o) \geqslant  \EE_o \big[ e^{\int_{0}^{t}\xi(W_{s})ds}; \; \eE_{t}^{*}\big] \approx e^{L^{*}t^{5/3}}
\end{equation*}
for all sufficiently large $t$. 

We should point out that the negligibility of the integral $\int_{0}^{\theta^{*}t}\xi(W_{s})ds$ is non-trivial. Intuitively, by requiring the Brownian motion $W$ to reach the neighbourhood of $z^t$ at time $\theta^{*} t$, $W$ is forced to travel at a speed much faster than its normal speed. As a result, it intends to travel like a geodesic with high probability. This makes it reasonable to replace $W|_{[0,\theta^{*}t]}$ by the corresponding geodesic $\gamma^{t}$ connecting $o$ and the center of $B^{t}$. Once this is possible, it is then not hard to see that the integral $\int_{0}^{\theta^{*}t}\xi(\gamma_{s}^{t})ds$ is negligible since the geodesic $\gamma^{t}$ rarely passes through very negative values of $\xi$.

\subsubsection{The upper asymptotics}

The proof of the matching upper asymptotics in \eqref{eq:intro_low_upper} is more challenging. We need to analyse all possible Brownian scenarios and show that none of them (nor their total contribution) produces a growth larger than $\exp ( L^* \, t^{5/3} )$. Here are the two concrete difficulties: 

\begin{itemize}
    \item In the candidate optimal scenario \eqref{eq:BM_scenario_optimal}, the proportion of time spent before the Brownian motion reaching the maximum value of $\xi$ is $\theta^* > 0$. The scenario \eqref{eq:BM_scenario_optimal} gives only one particular trajectory of the Brownian motion, but in general there are many other possible routes. 
    
    \item The exponential growth of volume in hyperbolic space gives exponentially many local peak regions of $\xi$ (as opposed to polynomially many in the Euclidean situation). This allows many more possibilities for $W$ to travel through these local peaks. 
\end{itemize}

Due to the above two aspects, a direct argument like the Euclidean situation would lead to an upper bound of the form $e^{L t^{5/3}}$ for some $L$ strictly larger than $L^*$. In order to get the precise constant $L^*$, we need to rule out all the above possibilities. 

At this stage, the essential ingredient is a fine analysis of the field $\xi$, showing that its clusters of peak regions are \textit{locally sparse} (in the sense Proposition~\ref{prop:local_maxima_sparse}, and then Lemma~\ref{lem:ClusterProperty} as a version adapted to our PAM problem). Together with a two-scale clustering argument, this ensures that all ``reasonable'' Brownian trajectories going to the maximum value of $\xi$ without too much cost must spend most of the traveling times outside peak regions of $\xi$. This then gives the precise constant $L^*$ as the upper bound. A full quantitative analysis is carried out in Section~\ref{sec:MainUp}

\subsection{Comparison with the Euclidean situation}
\label{sec:ComEuc}

Apart from the $t^{5/3}$-growth rate, the localisation mechanism
in the hyperbolic space is also quite different from the Euclidean
situation. In the Euclidean case, the almost sure growth of the Gaussian field is given by
\begin{equation*}
    \max_{|x| \leqslant R} \xi(x) \approx \sqrt{2 d \sigma^2 \log R} \quad \text{as} \; R \rightarrow \infty\;.
\end{equation*}
If one applies the same scenario as in \eqref{eq:BM_scenario} in the Euclidean situation, one is led to the lower bound
\begin{equation} \label{eq:Euclidean_low}
    u(t,o) \gtrsim \exp \bigg( \big( t-\theta(t) \big) \cdot \sqrt{2 d \,\sigma^{2} \, \log K(t)} - \frac{K^2(t)}{4 \theta(t)} \bigg)\;.
\end{equation}
The main difference with the hyperbolic situation is that we have the factor $\sqrt{\log K(t)}$ instead of $\sqrt{K(t)}$. The maximum for the quantity inside the exponential of the right hand side above turns out to be
\begin{equation} \label{eq:EucMax}
\sqrt{2 d \sigma^2} \cdot t \sqrt{\log t} + o(t\sqrt{\log t})\;,
\end{equation}
which precisely corresponds to the quenched growth in \eqref{eq:Euclidean}. It can be attained by taking
\begin{equation} \label{eq:EucLoc}
\theta(t) = \frac{t}{(\log t)^{\alpha}}\;, \quad K(t)=\frac{Kt}{(\log t)^{\beta}}
\end{equation}
for any $\alpha, \beta>0$ with $\alpha-2\beta < \frac{1}{2}$ and any $K>0$ independent of $t$. The interpretation of the leading order quantity in \eqref{eq:EucMax} is simple: the Brownian motion travels to a location of distance $\oO(t)$ away in a very short time, and stays nearby to pick up the maximum value of $\xi$ ($\approx \sqrt{2 d \sigma^2 \log t}$) for essentially all the time in $[0,t]$. However, there are a few key differences between the hyperbolic and Euclidean situations.

\vspace{2mm}\noindent (i)
While the Euclidean localisation mechanism is somewhat flexible (the range of parameters $\alpha$, $\beta$ and $K$ in \eqref{eq:EucLoc} are flexible), the optimal Brownian scenario in the hyperbolic situation is essentially unique. In fact, the right hand side of \eqref{eq:FormalLow} is maximised with $\theta(t) = \theta^* t$ and $K(t) = K^* t$ with the unique pair $(\theta^*, K^*)$ in \eqref{eq:optimal_values}. 

\vspace{2mm}\noindent (ii)
In the Euclidean case, the Brownian motion takes essentially no time ($\theta(t)/t\rightarrow 0$) to reach the location where $\xi$ takes its maximum. The cost (probability) of reaching there in time $\theta(t)$ is negligible compared to the reward, and does not contribute to the growth of $u(t,o)$. The constant $\sqrt{2 d \sigma^{2}}$ in the asymptotics \eqref{eq:Euclidean} only reflects the maximum of $\xi$. In the hyperbolic situation, the Brownian motion needs to spend a non-trivial proportion of time ($\theta^* t = t/5$) to enter a neighbourhood where $\xi$ reaches its peak near the boundary of $B(o,K^{*}t^{4/3})$. The cost traveling there turns out to be comparable with the reward from taking maximum of $\xi$. Hence, there is a competition between the gain from the $\xi$-maximum and the cost of traveling there. The exact growth constant $L^*$ is determined from the compromise between the two quantities.

\vspace{2mm}\noindent (iii) As a consequence of (ii), the upper bound techniques developed in \cite{CM95, GKM00} do not apply directly in the hyperbolic case. In fact, both arguments over there (which are very different indeed) are quite global and do not reflect the Brownian localisation at $\xi$-islands since its cost is negligible. In the hyperbolic case, such a localisation has to enter the upper bound argument in an essential way. In other words, we need
to compare every generic Brownian scenario to the optimal one and show that none of them produces a larger growth. 

\vspace{2mm}We conclude this section with a few remarks on our results. 

\begin{rem}
One can consider the hyperbolic space with curvature $-\kappa$ (instead of just having $\kappa=1$). The curvature constant $\kappa$ will be reflected in the Gaussian field growth as well as the heat kernel estimate which is used to compute probability costs. The aforementioned constants $L^*,K^*$ will both depend on $\kappa$. An interesting point is that the constant $\theta^*$ is always $1/5$ regardless of what $\kappa$ is. It is also interesting to note  that $L^*$ is non-linear with respect to the strength of the Gaussian field (its standard variance $\sigma$) in contrast to the Euclidean case.
\end{rem}

\begin{rem}
The proof of Theorem \ref{thm:MainThm} implicitly gives a localisation phenomenon for the Brownian motion under the Gibbs measure: with high (Gibbs) probability, $W$ reaches a particular location near $\partial B(o,K^{*}t^{4/3})$ at time $\theta^* t$ and stays there for the rest of time. On the other hand, it is not clear whether the PAM is \textit{completely localised}, i.e. does there exist a $\xi$-measurable site $Z(t)$ such that $v(t,Z_t)/u(t,o) \rightarrow 1$ in probability? Here $v$ denotes the solution to the PAM with $\delta_o$-initial condition and recall from Remark \ref{rem:DelIC} that $u(t,o)=\int_{\mathbb{H}^d}v(t,x)dx$.
\end{rem}

\begin{rem}
The choice of $\HH^d$ with constant negative curvature is mainly for technical convenience. The exponential volume growth plays an essential role, but in addition to that, the general mechanism should be quite robust with respect to fine geometric properties of the underlying space. Therefore, we expect that the methods developed in the current work can be utilised and extended to establish (sharp) asymptotic estimates for similar kind of models in more general geometric settings. 
\end{rem}
%

\subsection{Further questions}

The current article (along with the first part \cite{GX25}) provides
a simple initial attempt towards a deeper understanding on the fine
asymptotic behaviour of PAM in non-Euclidean settings. We mention a few natural questions for future investigation. 

\begin{enumerate}
\item \textit{Quenched fluctuation asymptotics.} Since the
precise localisation regime is identified in our methodology, the
fluctuation behaviour might be described in a suitably rescaled picture
with respect to the current localisation. It could be the case that
the fluctuation exponent is again determined by the Euclidean Laplacian
in a similar fashion as in the annealed asymptotics.

\item \textit{More singular Gaussian potentials.} In Euclidean situation, \cite{Che14, KPvZ22} revealed that singularity of the Gaussian potential already changes the quenched longtime behaviour of PAM (the power of $\log t$ in \eqref{eq:Euclidean} depends on the singularity of the noise). The case of white noise (in $d=2$ hyperbolic space) will be most interesting and challenging (the case for Euclidean situation was thoroughly studied in \cite{CvZ21, KPvZ22}). We do expect a different behaviour than Theorem~\ref{thm:MainThm} in the hyperbolic setting, and it is worth for further investigation. 

\item \textit{Time-dependent potentials} (e.g. white in time and coloured in space). In general, questions related to annealed and quenced asymptotics, fluctuation asymptotics, description of localisation vs delocalisation mechanism (in low vs high temperature) and the large scale behaviour of Brownian motion under Gibbs measure are widely open in continuous geometric settings. 
\end{enumerate}

\subsection*{Organisation of the article}

The rest of the article is organised as follows. In Section~\ref{sec:hyperbolic_BM}, we give preliminaries on the hyperbolic heat kernel estimates as well as some properties of the hyperbolic Brownian motion. In Section~\ref{sec:field_as}, we give almost sure properties of the random field $\xi$. All these will be used extensively in later sections. In Sections~\ref{sec:MainLow} and~\ref{sec:MainUp}, we give detailed proofs of the lower and upper asymptotics in \eqref{eq:intro_low_upper} respectively. They together imply Theorem~\ref{thm:MainThm}. Finally in Appendix~\ref{sec:LDP}, we prove a large deviations estimate that is used in the proof of the lower bound in Section~\ref{sec:MainLow}.

\subsection*{Notation}

We fix the base point (origin) of the hyperbolic space, denoted by $o$. For $z \in \HH^d$ and $r>0$, let $B(z,r)$ denote the $r$-neighbourhood of $z$ in $\HH^d$. We write
\begin{equation*}
    Q_r \triangleq B(o, r)\;.
\end{equation*}
We also write
\begin{equation} \label{eq:mu_0}
    \mu_0 \triangleq \sqrt{2(d-1) \sigma^2}\;,
\end{equation}
which is used throughout the article. We use $\P$ and $\E$ to denote the probability and expectation with respect to the randomness of $\xi$, and use $\PP_x$ and $\EE_x$ to denote the probability and expectation with respect to the randomness of the hyperbolic Brownian motion $W$ starting from $x \in \HH^d$. 

For two nonnegative functions $f$ and $g$, we write $f \asymp g $ if there exist universal constants $c, C>0$ such that
\begin{equation} \label{eq:equivalence_notation}
    c g(\cdot) \leqslant f(\cdot) \leqslant C g(\cdot)
\end{equation}
uniformly over the domain where $f$ and $g$ are defined.

\section{Hyperbolic Brownian motion and heat kernel}
\label{sec:hyperbolic_BM}

We list a few preliminary tools that will be used frequently in the sequel. The reader is referred to \cite[Section 2]{GX25} for a more general review. Throughout the rest of this article, we fix a base point
$o\in\mathbb{H}^{d}.$ We will use $Q_{R}$ to denote the geodesic ball of radius $R$ centered at $o$.

\subsubsection*{Volume form and heat kernel estimates}

The exponential volume growth of $Q_R$ as $R\rightarrow\infty$ plays an essential role in our study. The volume form on $\mathbb{H}^d$ under geodesic polar coordinates is given by
\begin{equation*}
    \mathrm{vol} \, (d \rho \, d \sigma) = (\sinh \rho)^{d-1} \, d \rho \, d \sigma\;.
\end{equation*}
The following volume estimate will be used throughout the article. 

\begin{lem} \label{lem:volume}
    We have the two sided estimates
    \begin{equation} \label{eq:volume_small}
        \vol (Q_r) \asymp_d r^d
    \end{equation}
    for $r \in (0,1)$, and
    \begin{equation} \label{eq:volume_large}
        \vol (Q_R) \asymp_d e^{(d-1) R}
    \end{equation}
    over all $R \geqslant 1$. The proportionality constants depend on the dimension $d$ only. 
\end{lem}
\begin{proof}
    These are immediate consequences of the expression of the volume form. 
\end{proof}

The following two-sided uniform estimate for the hyperbolic heat kernel will also be important to us (cf. \cite[Theorem 3.1]{DM98}).

\begin{lem}
\label{lem:HKEst}
Let $p(s,x,y)$ be the heat kernel on $\mathbb{H}^{d}$. We have the two-sided estimate
\begin{align} \label{eq:HKEst}
p(s,x,y) \asymp_d &\  s^{-\frac{d}{2}} \exp \Big[-\frac{(d-1)^{2}s}{4}-\frac{d^2 (x,y)}{4s} - \frac{(d-1) \, d(x,y)}{2} \Big]\nonumber \\
 & \ \ \ \times \big( 1+ d(x,y) +s \big)^{\frac{d-3}{2}} \, \big( 1 + d(x,y) \big),
\end{align}
where we recall from \eqref{eq:equivalence_notation} that $\asymp$ means both sides of \eqref{eq:HKEst} are bounded by a constant multiple of the other uniformly over $s>0$ and $x,y \in \HH^d$. 

In particular, we have
\begin{equation} \label{eq:HKEst_small}
    p(s,x,y) \asymp_d s^{-\frac{d}{2}} e^{-\frac{d^2(x,y)}{4s}}
\end{equation}
over all $s \leqslant 1$ and all $x,y \in \HH^d$ such that $d(x,y)$ is uniformly bounded, and that
\begin{equation} \label{eq:HKEst_large}
    e^{-\frac{d^2(x,y)}{4s} - c d(x,y)} \lesssim_d p(s,x,y) \lesssim_d e^{-\frac{d^2(x,y)}{4s} + c d(x,y)}
\end{equation}
uniformly over all $s \geqslant 1$ and $x,y \in \HH^d$ such that $d(x,y) \gg s \gg 1$. All proportionality constants depend on the dimension $d$ only. 
\end{lem}

\subsubsection*{Hyperbolic Brownian motion }

The Brownian motion on $\mathbb{H}^{d}$ is the Markov family $\{W_{t}^{x}:t\geqslant0\}$
($x\in M$ denotes its starting point) generated by $\Delta$. Its
behaviour differs from the Euclidean Brownian motion in several fundamental
ways. For instance, it travels a distance of order $t$ when $t$
is large (more precisely, $\frac{d(W_{t}^{x},x)}{(d-1)t}\rightarrow1$
a.s. as $t\rightarrow\infty$) and its angular component with respect
to the starting point $x$ converges to a definite limiting angle
(cf. \cite{Sul83}). In other words, the Brownian motion converges a.s. to a (random) limiting point on the boundary at infinity (which is homeomorphic to the sphere $S^{d-1}$).

The following lemma gives an SDE description for the radial process $R_u \triangleq d(W_u^{o},o)$. It is a direct consequence of the explicit expression of $\Delta$ under geodesic polar coordinates
with respect to $o$.

\begin{lem} \label{lem:RadialSDE}
The process $\rho_u \triangleq d(W_u^{o},o)$ satisfies the SDE 
\begin{equation} \label{eq:RadialSDE}
d \rho_u = \sqrt{2} d \beta_u + (d-1) \coth \rho_u du\;,
\end{equation}
where $\beta_{t}$ is a one-dimensional Euclidean Brownian motion. 
\end{lem}

%


The following proposition will be frequently used in Section~\ref{sec:MainUp}. 

\begin{prop} \label{prop:hitting_time}
    Let $\phi: \RR \rightarrow \RR$ be a non-increasing function. For every $\Theta > 0$, define the stopping time $\sigma = \sigma (\Theta)$ by
    \begin{equation*}
        \sigma := \inf \left\{ u>0: R_u = \Theta \right\}\;,
    \end{equation*}
    where $R_u \triangleq d(W_u^o, o)$. Then for every $\alpha \in (0,1)$, we have
    \begin{equation} \label{eq:hitting_time}
        \EE_o \Big[ \phi(\sigma) \1_{\{\sigma < s\}} \Big] \leqslant \frac{\alpha \Theta}{2 \sqrt{\pi}} \int_{0}^{s} u^{-3/2} e^{-\frac{\alpha^2 \Theta^2}{4u}} \phi (u) du
    \end{equation}
    for all sufficiently large $\Theta$ and $s$ (depending on $\alpha$) satisfying $\Theta >s \log s$. In particular, by taking $\phi \equiv 1$, we have
    \begin{equation} \label{eq:ExitEst}
        \PP_o \big( \sigma(\Theta) < s \big) \leqslant  e^{- c \Theta^2 / s}
    \end{equation}
    for some $c>0$ independent of $\Theta$ and $s$. 
\end{prop}
\begin{proof}
We give details for \eqref{eq:hitting_time}. The bound \eqref{eq:ExitEst} follows from \eqref{eq:hitting_time} immediately. Let $f:[0,\infty)\rightarrow\mathbb{R}$ be a non-decreasing, smooth function such that 
\begin{equation*}
    f(x)=\frac{1}{2} \; \text{on } \big[0,\frac{1}{3} \big]\;, \quad f(x)=x\ \text{on [1,\ensuremath{\infty)}}\;.
\end{equation*}
Hence, for $\Theta > 1$, $\sigma$ is also the hitting time of $\Theta$ for the process $f(R_u)$. We further assume $f$ is chosen such that $\|f'\|_{L^\infty} \leqslant 1$. Applying It\^o's formula to $f(\rho_u)$, we get 
\begin{align*}
f(\rho_u) & =\frac{1}{2} + \sqrt{2}\int_{0}^{u} f'(\rho_v) d \beta_v + \int_{0}^{u} \big((d-1)f'(\rho_v) \coth \rho_v + f''(\rho_v)\big) dv\\
 & =:\frac{1}{2} + \sqrt{2} M_u + Q_u\;.
\end{align*}
Note that $f'(\rho_v) = 0$ for $\rho_v < \frac{1}{2}$ and $\coth \rho_v$ is uniformly bounded for $\rho_v \geqslant \frac{1}{2}$. Hence, there exists $C>0$ depending on the choice of $f$ and dimension $d$ only such that $|Q_u| \leqslant C u$ for all $u >0$. As a consequence, for every $\alpha \in (0,1)$, we have
\begin{equation*}
    \PP \Big( \sup_{u \in [0,s]} f(\rho_u) > \Theta \Big) \leqslant \PP \Big( \sqrt{2} \sup_{u \in [0,s]} M_u > \Theta - \frac{1}{2} - C s \Big) \leqslant \PP \Big( \sqrt{2} \sup_{u \in [0,s]} M_u > \alpha \Theta \Big)
\end{equation*}
for all sufficiently large $\Theta \gg Cs$ (also depending on $\alpha$). Since the quadratic variation of $M$ satisfies
\begin{equation*}
    \langle M \rangle_u = \int_{0}^{u} \big( f'(\rho_v) \big)^2 dv \leqslant u\;,
\end{equation*}
and $(M_u)_{u \geqslant 0} \stackrel{\text{law}}{=} \big( B_{\langle M \rangle_u} \big)_{u \geqslant 0}$ for a standard Euclidean Brownian motion $B$ independent of $M$, we deduce
\begin{equation*}
    \PP \Big( \sqrt{2} \sup_{u \in [0,s]} M_u > \alpha \Theta \Big) = \widetilde{\PP} \Big( \sup_{u \in [0,s]} B_{\langle M \rangle_u} > \frac{\alpha \Theta}{\sqrt{2}} \Big) \leqslant \widetilde{\PP} \Big( \sup_{u \in [0,s]} B_u > \frac{\alpha \Theta}{\sqrt{2}} \Big)\;,
\end{equation*}
where $\widetilde{\PP}$ denotes the law of the standard Euclidean Brownian motion. In particular, this implies that the random variable
\begin{equation*}
    \widetilde{\sigma} := \inf \Big\{ u>0: B_u > \frac{\alpha \Theta}{\sqrt{2}} \Big\}
\end{equation*}
stochastically dominates $\sigma$ if $\Theta$ is sufficiently large, and hence
\begin{equation*}
    \EE \Big[ \phi(\sigma) \1_{\{\sigma < s\}} \Big] \leqslant \widetilde{\EE} \Big[ \phi(\widetilde{\sigma}) \1_{\{\widetilde{\sigma} < s\}} \Big]
\end{equation*}
since $\phi(\cdot) \1_{\{\cdot < s\}}$ is non-increasing. Finally, note that the density of the hitting time $\widetilde{\sigma}$ has the form
\begin{equation*}
    \widetilde{\PP} \big( \widetilde{\sigma} \in du \big) = \frac{\alpha \Theta}{2\sqrt{\pi}} \cdot u^{-3/2} \cdot e^{-\frac{\alpha^2 \Theta^2}{4u}} du\;.
\end{equation*}
The claim then follows. 
\end{proof}



\section{Almost sure asymptotics of $\xi$ on $\HH^d$}
\label{sec:field_as}

In this section, we develop some almost sure asymptotic behavior of the Gaussian field $\xi$ on $\HH^d$ that will be essentially used in the sequel. We first give the concept of (maximal) $r$-net.

\begin{defn} \label{def:net}
Let $E$ be a compact subset of a metric space $(X,d)$, and $r>0$. We say a finite set of points $\nN \subset E$ is an \textit{$r$-net} of $E$ if $d(z, z') > r$ for every $z, z' \in \nN$. 

$\nN$ is a \textit{maximal $r$-net} of $E$ if in addition, for any $x \in E$, there exists $z \in \nN$ such that $d(z, x) < r$. 
\end{defn}


\subsection{Sample function growth}
\label{subsec:XiGrowth}

We recall the following well-known result (cf. \cite[Theorem 2.1]{Adl90}).

\begin{lem} [Borell's inequality]
Let $\{X_{\alpha}: \alpha \in \aA\}$ be a centered Gaussian field with $\sup_{\alpha \in \aA} |X_\alpha| < +\infty$ almost surely. Then we have 
\begin{equation} \label{eq:Borell}
\P \Big(\sup_{\alpha \in \aA}|X_\alpha| > \Lambda \Big) \leqslant  4 \, \exp \bigg[ -\frac{1}{2 \sigma_\aA^2} \Big( \Lambda - \E \sup_{\alpha \in \aA} X_\alpha \Big)^2 \bigg]
\end{equation}
for all $\Lambda>\E[\sup_{\alpha \in A} X_\alpha]$, where $\sigma_\aA^{2}\triangleq\sup_{\alpha \in \aA}\E[X_\alpha^2]$. 
\end{lem}

The following lemma gives the growth of the Gaussian field $\xi$. It plays a prominent role in the proof of our main theorem.

\begin{lem} \label{lem:XiGrowth}
Recall $Q_{R}$ denotes the ball of radius $R$ centered at the origin $o$, and that $\mu_0 = \sqrt{2(d-1) \sigma^2}$ from \eqref{eq:mu_0}. Then $\P$-almost surely, we have
\begin{equation} \label{eq:XiUp}
\underset{R\rightarrow\infty}{\overline{\lim}} \, \Big( \frac{1}{\sqrt{R}} \max_{x \in Q_R} |\xi(x)| \Big) \leqslant \mu_0.
\end{equation}
\end{lem}
\begin{proof}
For each $R \geqslant 10$, let $\nN_R$ be a $1$-net of $Q_R$. By Definition~\ref{def:net}, the union $\cup_{z \in \nN_R} B(z,1)$ covers $Q_R$. Hence, we have
\begin{equation} \label{eq:Xi_growth_intermediate}
    \P \big(\max_{x \in Q_{R}} |\xi(x)|> \Lambda \big) \leqslant \sum_{z \in \nN_R} \P \big(\max_{x \in B(z, 1)} |\xi(x)|> \Lambda \big) \leqslant |\nN_R| \cdot \P \big(\max_{x\in Q_1} |\xi(x)|> \Lambda \big)\;,
\end{equation}
where $|\nN_R|$ denotes the cardinality of the set $\nN_R$, and the last equality follows from the stationarity of $\xi$. 

Again by Definition~\ref{def:net}, the balls $\big\{ B(z, 1/2) \big\}_{z \in \nN_R}$ are disjoint and are all contained in $Q_{R+1}$. This immediately implies $|\nN_R| \leqslant C_d \, e^{(d-1)R}$ for some constant $C_d$ depending on $d$ only. Now, applying Borell's inequality \eqref{eq:Borell} to the probability on the right hand side above and using $|\nN_R| \leqslant C_d \, e^{(d-1)R}$, we obtain
\begin{equation} \label{eq:XiGrowthPf1}
\P \big( \max_{x\in Q_R} |\xi(x)| > \Lambda \big) \leqslant 4 C_d \, \exp \Big( (d-1) R - \frac{(\Lambda - M)^2}{2 \sigma^2} \Big)
\end{equation}
for all $\Lambda > M \triangleq \E \big[\max_{x\in Q_1}\xi(x)\big]$. 

Let $\mu > \mu_0$ and take $\Lambda = \Lambda_R = \mu \sqrt{R}$, where $R$ is chosen sufficiently large so that $\mu_0 \sqrt{R} > M$. Then the exponent on the right hand side of \eqref{eq:XiGrowthPf1} satisfies
\begin{equation*}
    (d-1)R - \frac{1}{2\sigma^{2}} (\Lambda_R - M)^{2} = -\frac{1}{2\sigma^2} \big( \mu^2 - \mu_0^2 \big) R + O(\sqrt{R}).
\end{equation*}
It then follows from the Borel-Cantelli lemma that with probability one, we have
\begin{equation*}
    \max_{x \in Q_N} |\xi(x)| < \mu \sqrt{N}
\end{equation*}
for all sufficiently large integers $N$. For general $R$, we have
\begin{equation*}
    \max_{x \in Q_R} |\xi(x)| \leqslant \max_{x \in Q_{\lfloor R \rfloor + 1}} |\xi(x)| < \mu \sqrt{\lfloor R \rfloor + 1}
\end{equation*}
for all sufficiently large $R$, where $\lfloor R \rfloor$ denotes the largest integer less than or equal to $R$. This completes the proof of Lemma~\ref{lem:XiGrowth}. 
\end{proof}

The following lemma is also a direct consequence of  Borell's inequality
\eqref{eq:Borell}. 

\begin{lem} \label{lem:XiTail}
For each $r>0$, there exists a positive constant $c_r$ such that 
\begin{equation*}
    \PP \Big(\sup_{x\in Q_r} \xi(x) > \Lambda \Big) \leqslant e^{-c_r \Lambda^{2}}
\end{equation*}
for all $\Lambda>0.$ 
\end{lem}

\subsection{\label{subsec:XiPeak}Peaks near the boundary}

Next, we show Lemma~\ref{lem:XiGrowth} is sharp. More precisely, there is a (random) location $z_{*}$ near the boundary of $Q_{R}$ where $\xi$ maintains its peak value $\mu_0 \sqrt{R} + o(\sqrt{R})$ near $z_*$. We first introduce the following basic fluctuation estimate. 

\begin{lem} \label{lem:XiTube}
For every $a, b > 0$, we have
\begin{equation*}
    \P \Big( \max_{\|x\| \leqslant b \, h^{-1/3}} \big|\xi(x)-h \big|<a\sqrt{h} \Big) \geqslant \frac{1}{3} \cdot e^{-\frac{h^2}{2\sigma^2} + \frac{\sqrt{h}}{\sigma^2}}
\end{equation*}
for all sufficiently large $h$ (depending on $a$ and $b$), where we write $\|x\| = d(x,o)$ for simplicity. 
\end{lem}
\begin{proof}
This is essentially \cite[Lemma 3.3]{GKM00}. For every $h>0$, let $\P_h$ be the measure given by
\begin{equation} \label{eq:change_of_measure}
    \frac{d \P_{h}}{d \P} \triangleq e^{\frac{h}{\sigma^2} \cdot \xi(o) - \frac{h^2}{2 \sigma^2}}.
\end{equation}
Then under $\P_h$, the field $\xi_h(x) \triangleq \xi(x) - \frac{C( d(x,o) )}{\sigma^2} \cdot h$ has the same law as $\xi$ under $\P$ -- a centered Gaussian field with covariance function $C$. 

Note that the covariance function $C = C(\|x\|) = C\big( d(x,o) \big)$ is twice continuously differentiable at the origin, and satisfies $C(0) = \sigma^2$ and $C'(0) = 0$. Hence, there exists $L>0$ depending on the covariance function only such that
\begin{equation} \label{eq:field_translation}
    |\xi(x) - h| = \left| \xi_h(x) + \Big( \frac{C(\|x\|)}{\sigma^2} - 1 \Big) \cdot h \right| \leqslant |\xi_h(x)| + L \|x\|^2 h \leqslant |\xi_h(x)| + L \, b^2 h^{1/3}
\end{equation}
uniformly over $x$ with $\|x\| = d(x,o) < b h^{-1/3}$. Let
\begin{equation*}
    G_{h} = \Big\{ \max_{\|x\| \leqslant b \, h^{-1/3}} |\xi_h(x)| \leqslant a \sqrt{h} / 2 \Big\}\;.
\end{equation*}
Note that $G_h$ also depends on the parameters $a$ and $b$, but we omit them for notational simplicity. Take $h$ sufficiently large such that $L b^2 h^{1/3} < \frac{a \sqrt{h}}{2}$. Then by \eqref{eq:field_translation} and the change-of-measure \eqref{eq:change_of_measure}, we have
\begin{equation*}
    \P \Big( \max_{\|x\| \leqslant b \, h^{-1/3}} \big|\xi(x)-h \big|<a\sqrt{h} \Big) \geqslant \P ( G_h ) = e^{-\frac{h^2}{2 \sigma^2}} \int_{G_h} e^{-\frac{h}{\sigma^2} \cdot \xi_h(o)} d \P_h\;.
\end{equation*}
By further restricting to the event $\{\xi_h(o) < - 1/ \sqrt{h}\}$ and using $\text{Law}_{\P_h} (\xi_h) = \text{Law}_\P (\xi)$, we can bound the integral on the right hand side above (from below) by
\begin{equation*}
    \begin{split}
    \int_{G_h} e^{-\frac{h}{\sigma^2} \cdot \xi_h(o)} d \P_h &\geqslant e^{\frac{\sqrt{h}}{\sigma^2}} \cdot \P_h \big( G_h \cap \{\xi_h(o) < - 1/\sqrt{h}\} \big)\\
    &\geqslant e^{\frac{\sqrt{h}}{\sigma^2}} \cdot \Big( \P_h \big( \xi_h(o) < - 1/ \sqrt{h} \, \big) - \P_h (G_h^c) \Big)  \geqslant \frac{1}{3} \, e^{\frac{\sqrt{h}}{\sigma^2}}
    \end{split}
\end{equation*}
for sufficiently large $h$. The claim then follows. 
\end{proof}

The main result of this subsection is the following. 

\begin{prop} \label{prop:XiPeak}
With $\P$-probability one, for every $b>0$ and all sufficiently large $R$, there exists $z_*$ with $d(z_*, \, o) \in (R-\sqrt{R}, R-1)$ such that 
\begin{equation} \label{eq:XiPeak}
\min_{x \in B(z_*, \, b R^{-1/6})} \xi(x) \geqslant \mu_0 \sqrt{R} - 2 ( \mu_0^2 R )^{1/4}\;.
\end{equation}
\end{prop}
\begin{proof}
Recall that $r_0$ is the correlation length of the field $\xi$. For each $R> r_0$, let $\nN_{2 r_0}^R$ be a maximal $(2 r_0)$-net of the annulus
\begin{equation*}
    \aA_R \triangleq \big\{ z \in \HH^d: \; R - \sqrt{R} < d(z,o) < R-1 \big\}\;.
\end{equation*}
Consider the event
\begin{equation*}
    \eE_R \triangleq \bigcup_{z \in \nN_{2 r_0}^R} \left\{ \min_{x \in B(z, \, b R^{-1/6})} \xi(x) > \mu_0 \sqrt{R} - 2 ( \mu_0^2 R )^{1/4}\right\}.
\end{equation*}
Since points in $\nN_{2 r_0}^R$ are at least $2 r_0$ away from each other, the events in the union above are all independent. Furthermore, they all have the same probability by stationarity of $\xi$. Hence, the probability of the complement event $\eE_R^c$ satisfies
\begin{equation} \label{eq:PfKeyLem2}
    \begin{split}
    \P(\eE_R^c) &= \bigg( 1 - \P \Big( \min_{\|x\| \leqslant b \, R^{-1/6}} \xi(x) >\mu_0 \sqrt{R} - 2 (\mu_0^2 R)^{\frac{1}{4}} \Big) \bigg)^{|\nN_{2 r_0}^R|}\\
    &\leqslant \exp \bigg( - |\nN_{2 r_0}^R| \cdot \P \Big( \min_{\|x\| \leqslant b \, R^{-1/6}} \xi(x) >\mu_0 \sqrt{R} - 2 (\mu_0^2 R)^{1/4} \Big) \bigg)\;,
    \end{split}
\end{equation}
where we have used $1-q \leqslant e^{-q}$ in the second line above. We now control from below the two terms inside the exponential above. For the probability term, by Lemma~\ref{lem:XiTube} with $h = \mu_0 \sqrt{R}$, we have
\begin{equation} \label{eq:PfKeyLem2'}
    \begin{split}
    &\phantom{111}\P \Big( \min_{x \in B(o, \, b \, R^{-1/6})} \xi(x) >\mu_0 \sqrt{R} - 2 (\mu_0^2 R)^{1/4} \Big)\\
    &\geqslant \P \Big( \max_{x \in B(o, \, b \, R^{-1/6})} |\xi(x) - \mu_0 \sqrt{R}| <2 (\mu_0^2 R)^{1/4} \Big) \gtrsim e^{-(d-1) R + c R^{1/4}}\;,
    \end{split}
\end{equation}
where we used $\mu_0^2 = 2 (d-1) \sigma^2$ to get the precise term $-(d-1) R$ on the exponential. As for the cardinality of $\nN_{2 r_0}^R$, by Definition~\ref{def:net}, $\cup_{z \in \nN_{2 r_0}^R} B(z, 2 r_0)$ covers the annulus $\aA_R$. Hence, we have the lower bound
\begin{equation} \label{eq:PfKeyLem3}
    |\nN_{2 r_0}^R| \geqslant {\rm vol}(\aA_R) / {\rm vol}(Q_{2 r_0}) \geqslant c' e^{(d-1) R}\;,
\end{equation}
where we also used the volume estimate \eqref{eq:volume_large} from Lemma~\ref{lem:volume}. Plugging \eqref{eq:PfKeyLem2'} and \eqref{eq:PfKeyLem3} back into \eqref{eq:PfKeyLem2} and taking exponential, we obtain
\begin{equation*}
    \P \big( \eE_R^c \big) \leqslant \exp \big( - c_1 e^{c_2 R^{1/4}} \big)
\end{equation*}
for some $c_1, c_2 > 0$. By Borel-Cantelli lemma, we conclude that $\P$-almost surely, $\eE_R$ happens for all large $R$. This concludes the proof of the lemma. 
\end{proof}

\subsection{Local peaks}

Next, we discuss aspects about the geometry of local peaks of the field $\xi$. For $\delta, \eta, R>0$, $z \in Q_R$ and $L \in \NN$, let $\eE_R^z (\delta, \eta, L)$ be the event
\begin{equation*}
    \begin{split}
    \eE_R^z &(\delta, \eta, L) \triangleq \Big\{ \; \exists \; y_1, \dots, y_L \in B(z,\eta R) \cap Q_R \; \; \text{such that}\\
    &|\xi(y_j)| > \delta \sqrt{R} \; \text{for all} \; 1 \leqslant j \leqslant L, \; \text{and} \; d(y_i, y_j) > 9 r_0 \; \text{for all} \; i \neq j \Big\}\;,
    \end{split}
\end{equation*}
and
\begin{equation*}
    \eE_R (\delta,\eta,L) \triangleq \bigcup_{z \in Q_R} \eE_R^z (\delta,\eta,L)\;.
\end{equation*}
We have the following proposition. 

\begin{prop} \label{prop:local_maxima_sparse}
    There exists $\delta_0 > 0$ such that for every $\delta \in (0, \delta_0)$, $\eta < \delta^3$ and $L > \delta^{-3}$, we have
    \begin{equation*}
        \P \Big( \underset{R \rightarrow \infty}{\overline{\lim}} \eE_R(\delta,\eta,L) \Big) = 0\;.
    \end{equation*}
\end{prop}
\begin{proof}
    Let $\nN_{\eta R}^R$ be a maximal $(\eta R)$-net of $Q_R$. By definition, for every $z \in Q_R$, there exists $z' \in \nN^R_{\eta R}$ such that $B(z, \eta R) \subset B(z', 2 \eta R)$. Hence, we have
    \begin{equation} \label{eq:freezing_ball}
        \P \big( \eE_R (\delta, \eta, L) \big) \leqslant \sum_{z' \in \nN_{\eta R}^R} \P \big( \eE_R^{z'} (\delta, 2\eta, L) \big) \leqslant C_\eta \; e^{(d-1) R} \; \P \big( \eE_R^o (\delta, 2 \eta, L) \big)\;,
    \end{equation}
    where we also used $|\nN_{\eta R}^R| \leqslant C_\eta \, e^{(d-1)R}$ as well as stationarity of $\xi$ in the last step. 
    
    We now turn to estimating the quantity $\P \big( \eE_R^o (\delta, 2 \eta, L) \big)$. Let $\nN_{r_0}$ be a maximal $r_0$-net of $Q_{2 \eta R}$. Then we necessarily have $|\nN_{r_0}| \leqslant C_{\eta} e^{2 \eta (d-1) R}$. Suppose there exist $y_1, \dots, y_L \in Q_{2 \eta R}$ that are $9 r_0$-away from each other and that $|\xi(y_j)| > \delta \sqrt{R}$ for all $j$. Then for each $j = 1, \dots, L$, there exists $x_j \in \nN_{r_0}$ such that $y_j \in B(x_j, r_0)$. By the assumption that the points $\{y_j\}_{j=1}^{L}$ are $9 r_0$ away from each other, we see the points $x_1, \dots, x_L$ are $5 r_0$ away from each other. As a consequence, the balls $B(x_j, r_0)$ are all $3 r_0$ away from each other, and hence the field $\xi$ restricted to these balls are all independent. 

    By Lemma~\ref{lem:XiTail} and the above independence, there exists $c = c(r_0)$ such that
    \begin{equation*}
        \P \Big( \max_{z \in B(x_j, r_0)} |\xi(z)| > \delta \sqrt{R} \; \text{for all} \; j = 1, \dots, L \Big) \lesssim e^{- c L \delta^2 R}\;.
    \end{equation*}
    This implies
    \begin{equation*}
        \P \Big( \eE_R^o (\delta, 2 \eta, L) \Big) \lesssim \begin{pmatrix} |\nN_{r_0}| \\ L \end{pmatrix} e^{-c L \delta^2 R} \lesssim \exp \Big( \big( 2(d-1) \eta - c \delta^2 \big) L R \Big)\;.
    \end{equation*}
    Plugging the bound back into \eqref{eq:freezing_ball}, taking $L = \delta^{-3}$ and making $\delta_0$ sufficiently small, we get
    \begin{equation*}
        \P \big( \eE_R^o (\delta, 2\eta, L) \big) \lesssim e^{- c \delta^{-1} R}
    \end{equation*}
    for a possibly different $c>0$. The conclusion of the proposition then follows from the Borel-Cantelli lemma. 
\end{proof}

\begin{rmk} \label{rmk:local_maxima_sparse}
    Loosely speaking, Proposition~\ref{prop:local_maxima_sparse} says that if $\delta$ is sufficiently small and $\eta < \delta^3$, then $\P$-almost surely, for all sufficiently large $R$, no ball of radius $\eta R$ inside $Q_R$ can have more than $\delta^{-3}$ points in $\mM_\delta^R$ which are $9 r_0$ separated from each other. 
\end{rmk}

\section{The lower asymptotics}
\label{sec:MainLow}

In this section, we prove the precise lower asymptotics of $u(t,o)$. Recall the growth constant $L^*$ from \eqref{eq:OptimalExp}. The main statement is the following. 

\begin{thm} \label{thm:MainLow}
For every $\theta \in [0,1]$ and $K>0$, we have
\begin{equation} \label{eq:MainLow}
\underset{t\rightarrow\infty}{\underline{\lim}} \frac{1}{t^{5/3}} \log u(t,o) \geqslant L^*
\end{equation}
for $\P$-a.e. realisation of $\xi$. 
\end{thm}

We will show that for every $\theta \in [0,1]$ and $K>0$, we have
\begin{equation} \label{eq:MainLow_intermediate}
    \underset{t\rightarrow\infty}{\underline{\lim}}\frac{1}{t^{5/3}} \log u(t,o) \geqslant (1-\theta) \mu_0 \sqrt{K} - \frac{K^2}{4 \theta}\;.
\end{equation}
Theorem~\ref{thm:MainLow} then follows immediately from the definition of $L^*$ in \eqref{eq:OptimalExp} and that $\theta \in [0,1]$ and $K>0$ can be arbitrary. The proof of \eqref{eq:MainLow_intermediate} is to identify one Brownian scenario under which the growth exponent is $t^{5/3}$ and the constant matches the right hand side of \eqref{eq:MainLow_intermediate}. 

We fix $\theta \in [0,1]$ and $K>0$ arbitrary. We also fix a realisation of the Gaussian field $\xi$ and $T$ depending on the realisation such that all statements in Section~\ref{sec:field_as} hold for $R > K T^{4/3}$. We also set the scale function
\begin{equation} \label{eq:zeta}
    \zeta(t) \triangleq t^{-2/9}\;.
\end{equation}
This will be used as for the radius of the ball of peak values of $\xi$ near the boundary $Q_R$ for $R = R(t) \sim t^{4/3}$. Hence, according to Proposition~\ref{prop:XiPeak}, its radius is of order $R^{-1/6} \sim t^{-2/9}$. 

The rest of this section is as follows. In Section~\ref{sec:localisation_maximum}, we describe the Brownian scenario where we localise the Feynman-Kac formula around, and obtain a first lower bound for $u(t,o)$. In Section~\ref{sec:Low_proof_ingredients}, we give several other key ingredients and give a complete proof of Theorem~\ref{thm:MainLow} by first assuming these ingredients. In later subsections, we turn back and prove these ingredients.

\subsection{Localisation and picking up the maximum after $\theta t$}
\label{sec:localisation_maximum}

By Proposition~\ref{prop:XiPeak}, for every $t>T$, there exists $z^t \in \HH^d$ with
\begin{equation} \label{eq:location_z}
    K t^{4/3} - \sqrt{K} t^{2/3} < d(z^t, o) < K t^{4/3} - 1
\end{equation}
such that
\begin{equation} \label{eq:Xi_peak_lower}
    \min_{x \in B(z^t, \, 2 \zeta(t))} \xi(x) \geqslant \mu_0 \sqrt{K} t^{2/3} - 2 (\mu_0^2 K)^{\frac{1}{4}} t^{1/3}\;,
\end{equation}
where we recall from \eqref{eq:zeta} that $\zeta (t) = t^{-2/9}$. 

Let $\gamma^t: [0, \theta t] \rightarrow \HH^d$ be the geodesic parametrised at uniform speed with starting and end points $\gamma^t_0 = o$ and $\gamma^t_{\theta t} = z^t$. Fix $\eps>0$ small. Let $C_\eps(t)$, $E(t)$ and $S(t)$ be the events
\begin{equation} \label{eq:lower_bound_events}
    \begin{split}
    C_\eps(t) & \triangleq \Big\{  \sup_{s \in [0, \theta t]} d (W_s, \gamma^t_s) < \eps \Big\}\;, \quad E(t) \triangleq \Big\{W_{\theta t}\in B(z^{t},\zeta(t)) \Big\}\;,\\
    S(t) & \triangleq \Big\{W([\theta t,t]) \subseteq B(z^{t},2\zeta(t))  \Big\}\;.
    \end{split}
\end{equation}
They correspond to ``concentrating in the $\eps$-neighbourhood of the geodesic $\gamma^{t}$ in $[0, \theta t]$'', ``entering the ball $B(z^{t},\zeta(t))$ at time $\theta t$'' and ``staying inside $B(z^{t},2\zeta(t))$ in $[\theta t, t]$'' respectively. The role of $C_\eps(t)$ is to allow one to replace the Brownian trajectory by the geodesic $\gamma^{t}$, along which the integral of $\xi$ over $[0, \theta t]$ is easier to estimate. The role of $S(t)$, as we explained in Section \ref{subsec:StratLow}, is to let the Brownian motion stay at the peak region of $\xi$ to pick up its approximate maximum. 

Localising the Feynman-Kac formula \eqref{eq:FKIntro} on the event $C_\eps(t) \cap E(t) \cap S(t)$, and conditioning on ${\cal F}_{\theta t}$, we have
\begin{equation*}
u(t,o) \geqslant \mathbb{E}_{o} \bigg[e^{\int_{0}^{\theta t}\xi(W_{s})ds}{\bf 1}_{C_\eps(t) \cap E(t)} \cdot \mathbb{E} \Big(e^{\int_{\theta t}^{t} \xi(W_{s}) ds}{\bf 1}_{S(t)} | {\cal F}_{\theta t} \Big) \bigg].
\end{equation*}
By \eqref{eq:Xi_peak_lower} and definition of the event $S(t)$ in \eqref{eq:lower_bound_events}, there exists $c_0>0$ depending on $\mu_0$ and $K$ only such that
\begin{equation*}
    \int_{\theta t}^{t} \xi(W_{s}) ds \geqslant \; (1-\theta) \mu_0 \sqrt{K} \, t^{5/3} - c_0 t^{4/3} \quad \text{on} \; \; S(t)
\end{equation*}
for all sufficiently large $t$. Plugging it back to the above lower bound for $u(t,o)$, we get
\begin{equation} \label{eq:u_low_maximum}
    u(t,o) \geqslant e^{(1-\theta) \mu_0 \sqrt{K} t^{5/3} - c_0 t^{4/3}} \EE_o \Big[ e^{\int_{0}^{\theta t} \xi(W_s) ds} {\bf 1}_{C_\eps(t) \cap E(t)} \cdot \PP_{W_{\theta t}} \big( S(t) \big) \Big]
\end{equation}
for all large $t$. Our main task is now reduced to estimating the quantities appearing on the right hand side above.

\subsection{Proof of Theorem~\ref{thm:MainLow}}
\label{sec:Low_proof_ingredients}

In this section, we state propositions that control from below various parts in the expectation term on the right hand side of \eqref{eq:u_low_maximum}. Assuming these ingredients, we combine them together to prove Theorem~\ref{thm:MainLow}. The proofs of these ingredients are postponed to later subsections. 

The first ingredient is the probability of Brownian motion staying in $B(z^t, 2 \zeta(t))$ in time interval $[\theta t, t]$, conditioning on it entering $B(z^t, \zeta(t))$ at time $\theta t$. We give it in the following lemma. 

\begin{prop} \label{prop:StayProb}
    There exist universal constants $c_1, c_2 > 0$ such that
    \begin{equation*}
        \1_{E(t)} \cdot \PP_{W_{\theta t}} \big( S(t) \big) \geqslant c_1 \1_{E(t)} \cdot e^{-\frac{c_2 (1-\theta) t}{\zeta^2(t)}}
    \end{equation*}
    for all $t \geqslant 1$, where we recall $\zeta(t) = t^{-2/9}$ from \eqref{eq:zeta}, and definitions of the events $E(t)$ and $S(t)$ from \eqref{eq:lower_bound_events}. 
\end{prop}

Assuming this lemma for the moment, and applying it to $\1_{E(t)} \PP_{W_{\theta t}} \big( S(t) \big)$ on the right hand side of \eqref{eq:u_low_maximum} (recalling that $\zeta(t) = t^{-2/9}$), we get
\begin{equation} \label{eq:u_low_staying}
    u(t,o) \gtrsim e^{(1-\theta) \mu_0 \sqrt{K} t^{5/3} - c (t^{4/3} + t^{13/9})} \cdot \EE_o \Big[ e^{\int_{0}^{\theta t} \xi(W_s) ds} \, \1_{C_\eps(t) \cap E(t)} \Big]
\end{equation}
for some $c>0$ depending on $\theta$ and $K$, but independent of $\eps$ and $t$. 

The next statement controls the integral of $\xi$ on $[0, \theta t]$ from below on the event $C_\eps(t)$. 

\begin{prop} \label{prop:FormIntEst}
Let $\eps>0$ be as above and $\delta < \delta_0$ sufficiently small as in Proposition~\ref{prop:local_maxima_sparse}. There exists a deterministic constant $C>0$ independent of $\eps$, $\delta$ and $t$ such that for $\P$-a.e. realisation of $\xi$, we have
\begin{equation} \label{eq:FormIntEst}
\int_{0}^{\theta t} \xi(W_{s})ds \geqslant - C \, \left( \frac{t^{1/3}}{\delta^6} + (\eps + \delta) \, t^{5/3} \right) \quad \text{on} \; \; C_\eps(t)
\end{equation}
for all sufficiently large $t$, all $\eps>0$ and $\delta < \delta_0$ (where $\delta_0$ is as in Proposition~\ref{prop:local_maxima_sparse}). 
\end{prop}

Note that the expectation term on the right hand side of \eqref{eq:u_low_staying} is localised on $C_\eps(t)$. Hence, we can apply Proposition~\ref{prop:FormIntEst} to that expectation term to deduce
\begin{equation} \label{eq:u_low_early_integral}
    \begin{split}
    u(t,o) \gtrsim \exp \Big[ \big( &(1-\theta) \mu_0 \sqrt{K} - (\eps + \delta) \big) t^{5/3} - c \big( t^{4/3} + t^{13/9} + t^{1/3} / \delta^6 \, \big) \Big]\\
    &\times \PP_o \big( C_\eps(t) \cap E(t) \big)\;,
    \end{split}
\end{equation}
where the constant $c>0$ is again independent of $t$, $\eps$ and $\delta$ (but allowed to depend on $\theta$ and $K$). We now give the last ingredient -- a lower bound for $\PP_o \big( C_\eps(t) \cap E(t) \big)$. 

\begin{prop} \label{prop:probability_concentrating}
    There exists $c>0$ independent of $t$ (but depending on $\theta$, $K$ and $\eps$) such that
    \begin{equation*}
        \PP_o \big( C_\eps(t) \cap E(t) \big) \gtrsim \exp \Big( - \frac{K^2}{4 \theta} \cdot t^{5/3} - c \, t^{13/9} \Big)
    \end{equation*}
    for all sufficiently large $t$. 
\end{prop}

Now, applying Proposition~\ref{prop:probability_concentrating} to the probability term on right hand side of \eqref{eq:u_low_early_integral}, taking logarithm and sending $t \rightarrow +\infty$, we obtain
\begin{equation*}
    \underset{t\rightarrow\infty}{\underline{\lim}} \frac{1}{t^{5/3}} \log u(t,o) \geqslant (1-\theta) \mu_0 \sqrt{K} - \frac{K^2}{4 \theta} - C(\eps +\delta)\;,
\end{equation*}
which holds for all $\theta \in [0,1]$, $K>0$ and $\eps, \delta > 0$ sufficiently small. Hence, taking $\eps, \delta \rightarrow 0$ and optimising over $\theta \in [0,1]$ and $K>0$ gives Theorem~\ref{thm:MainLow}.

\subsection{Probability of staying inside a small ball -- proof of Proposition~\ref{prop:StayProb}}
\label{subsec:StayProb}

Note that on the event $E(t)$, we have
\begin{equation*}
    \PP_{W_{\theta t}} \big( S(t) \big) \geqslant \inf_{x \in B(z^t, \zeta(t))} \PP \big( S(t) | W_{\theta t} = x \big) \geqslant \PP_o \big( W |_{[0, (1-\theta) t]} \subset Q_{\zeta(t)} \big)\;.
\end{equation*}
The proof of Proposition~\ref{prop:StayProb} is complete once we get the following statement. 

\begin{lem}
There exist universal constants $C_1, C_2>0$ such that
\begin{equation} \label{eq:StayProb}
\PP_o \big( W |_{[0, (1-\theta) t]} \subset Q_{\zeta (t)} \big) \geqslant C_1 e^{- C_2 \frac{(1-\theta)t}{\zeta(t)^{2}}}
\end{equation}
for all $t \geqslant 1$. 
\end{lem}
\begin{proof}
Define $F: \RR^+ \rightarrow \RR^+$ by $F(r) = \coth (r)$ for $r \leqslant \frac{\zeta(t)}{4}$ and $F(r) = \frac{8}{\zeta(t)}$ for $r > \frac{\zeta(t)}{4}$. Since $\coth (r)$ is decreasing in $r$ and
\begin{equation*}
    \coth (r) = \frac{e^{2r} + 1}{e^{2r} - 1} = 1 + \frac{2}{e^{2r}-1} \leqslant 1 + \frac{1}{r}\;,
\end{equation*}
we see $F(r) \geqslant \coth (r)$ for all $r \geqslant 0$. Define the process $(\tilde{R}_s)_{s \geqslant 0}$ by
\begin{equation*}
    d \tilde{R}_s =  (d-1)  F(\tilde{R}_s) ds + \sqrt{2} \, d \beta_s\;, \quad \tilde{R}_0 = 0\;,
\end{equation*}
where $\beta$ is the standard Euclidean Brownian motion. Comparing the above equation with that for $R_s := d(W_s, o)$ in Lemma~\ref{lem:RadialSDE}, we see
\begin{equation*}
    \mathbb{P}_o \big( W([0,(1-\theta) t]) \subset Q_{\delta (t)} \big) \geqslant \mathbb{P} \Big( \sup_{s \in [0,(1-\theta)t]} \tilde{R}_s < \zeta(t) \Big)\;,
\end{equation*}
where with an abuse of notation, $\PP$ on the right hand side above denotes that for the standard Euclidean Brownian motion $\beta$. Let
\begin{equation} \label{eq:event_staying}
    \tau \triangleq (1-\theta) t \; \wedge \; \inf_{s > 0} \big\{\tilde{R}_s = \zeta(t)/2 \big\}\;.
\end{equation}
Since $F(r) =  8/ \zeta(t) $ for $r \in [\zeta(t)/4, \, 3\zeta(t)/4]$, then
\begin{equation*}
    \begin{split}
    \tilde{R}_s - \tilde{R}_\tau &= \frac{8(d-1)}{\delta} \cdot (s-u) + \sqrt{2} (\beta_s - \beta_\tau)\\
    &= \bigg( \frac{8(d-1)}{\zeta(t)} \cdot s + \sqrt{2} \beta_s \bigg) - \bigg( \frac{8(d-1)}{\zeta(t)} \cdot \tau + \sqrt{2} \beta_\tau \bigg)
    \end{split}
\end{equation*}
for all time $s>\tau$ and before $\tilde{R}$ hitting back on $\frac{\zeta(t)}{4}$. This implies if the trajectory $\beta$ satisfies
\begin{equation} \label{eq:event_Euclidea_BM}
     \bigg\{ \sup_{s \in [0, (1-\theta)t]} \Big| \sqrt{2} \beta_s + \frac{8(d-1)}{\zeta(t)} \cdot s \Big| < \frac{\zeta(t)}{8} \bigg\}\;,
\end{equation}
the process $\tilde{R}$ will stay in the interval $[\zeta(t)/4, 3\zeta(t)/4]$ after the hitting time $\tau$ and until time $(1-\theta) t$. In particular, $\tilde{R}$ never exceeds $\zeta(t)$ before time $(1-\theta) t$ under the event \eqref{eq:event_Euclidea_BM}. 

It then remains to compute the probability of the event \eqref{eq:event_Euclidea_BM}. By scaling, we have
\begin{equation*}
    \mathbb{P} \Big( \sup_{s \in [0, (1-\theta)t]} \Big| \sqrt{2} \beta_s + \frac{8(d-1)}{\zeta(t)} \cdot s \Big| < \frac{\zeta(t)}{8} \Big) = \mathbb{P} \Big( \sup_{s \in [0, \, c (1-\theta)t/\delta^2(t)]} |\beta_s - C s| < 1 \Big)
\end{equation*}
for some $c, C > 0$ independent of $t$. The claim then follows from standard large deviation estimates for Euclidean Brownian motion. 
\end{proof}

\subsection{Analysis of $\xi$-integral on $[0, \theta t]$ -- proof of Proposition~\ref{prop:FormIntEst}}

This and next subsection are devoted to controlling the expectation 
\begin{equation*}
    \EE_o \Big[ e^{\int_{0}^{\theta t} \xi(W_s) ds} \1_{C_\eps(t) \cap E(t)} \Big]\;.
\end{equation*}
This subsection focuses on the behaviour of the integral on the event $C_\eps(t)$, while in the next subsection we estimate the probability of the event $C_\eps(t) \cap E(t)$. 

Recall from \eqref{eq:location_z} the point $z^t$ and from Section~\ref{sec:localisation_maximum} that $\gamma^t: [0, \theta t] \rightarrow \HH^d$ is the geodesic connecting the origin $o$ and $z^t$ and parametrised at uniform speed. Write
\begin{equation*}
    R(t) \triangleq d (z^t, o) = K t^{4/3} + \oO(t^{2/3})\;.
\end{equation*}
Fix $\Lambda \in (9 r_0, 10 r_0)$ so that $N = N(t) = \frac{R(t)}{\Lambda}$ is an integer. Partition the time interval $[0, \theta t]$ into $N$ pieces of equal length as
\begin{equation} \label{eq:time_interval_partition}
    0 = t_0 < t_1 < \cdots < t_{N-1} < t_N = \theta t\;, \qquad t_k = k \theta \Lambda \cdot \frac{t}{R(t)}\;.
\end{equation}
The choice of $\Lambda$ depends on $t$ since one needs $N$ to be integer, but we omit it in notation since it has a universal lower and upper bound independent of $t$ and $\eps$. By \eqref{eq:location_z}, the length of each interval is
\begin{equation} \label{eq:time_interval_length}
    t_k - t_{k-1} = \theta \Lambda \cdot \frac{t}{R(t)} = \frac{\theta \Lambda}{K} \cdot t^{-1/3} + \oO(t^{-1})\;.
\end{equation}
The length of the geodesic segment of $\gamma^t$ restricted to each of these intervals is $\Lambda$. We are now ready to prove Proposition~\ref{prop:FormIntEst}. 

\begin{proof} [Proof of Proposition~\ref{prop:FormIntEst}]
Write
\begin{equation} \label{eq:integral_early_split}
    \int_{0}^{\theta t} \xi(W_s) ds = \int_{0}^{\theta t} \xi(\gamma_s^t) dx + \int_{0}^{\theta t} \big( \xi (W_s) - \xi(\gamma_s^t) \big) ds\;,
\end{equation}
and we separately control the two terms, starting from the first one with the geodesic $\gamma^t$ as input. 

Let $\delta_0$ be as in Proposition~\ref{prop:local_maxima_sparse}, and $\delta < \delta_0$ arbitrary. Each piece of trajectory $\gamma^t|_{[t_{k-1},t_k]}$ is a geodesic segment of length $\Lambda > 9 r_0$. Every consecutive $\frac{\delta^3 R(t)}{\Lambda}$ such geodesic segments forms a geodesic segment of length $\delta^3 R(t)$, and hence is contained in a ball of radius $\delta^3 R(t)$. Consider the set of integers
\begin{equation*}
    \aA \triangleq \Big\{ 1 \leqslant k \leqslant N\,: \; \inf_{u \in [t_{k-1}, t_k]} \xi(\gamma^t_u) < - \delta \sqrt{R(t)} \Big\}\;.
\end{equation*}
By Proposition~\ref{prop:local_maxima_sparse} (see also Remark~\ref{rmk:local_maxima_sparse}), $\P$-almost surely, for sufficiently large $t$, for every consecutive $\frac{\delta^3 R(t)}{\Lambda}$ integers, its intersection with $\aA$ has cardinality at most $\delta^{-3}$. Note that one needs at most $\delta^{-3}$ such consecutive integer sequence of cover $[1, N]$. As a consequence, the cardinality of $\aA$ satisfies $|\aA| \leqslant \delta^{-6}$. Hence, we have
\begin{equation} \label{eq:integral_early_1}
    \begin{split}
    \int_{0}^{\theta t} \xi(\gamma_s^t) ds &= \sum_{k \in \aA} \int_{t_{k-1}}^{t_k} \xi(\gamma_s^t) ds + \sum_{k \in \aA^c} \int_{t_{k-1}}^{t_k} \xi(\gamma^t_s) ds\\
    &\geqslant - 2 |\aA| \cdot \mu_0 \sqrt{K} t^{2/3} \cdot \frac{\theta \Lambda}{K} \big( t^{-1/3} + \oO(t^{-1}) \big) - \theta t \cdot \delta \sqrt{K} \cdot t^{2/3}\\
    &\geqslant - c \Big( \frac{t^{1/3}}{\delta^6} + \delta \, t^{5/3} \Big)\;.
    \end{split}
\end{equation}
Here in the inequality on the second line, we use the crude lower bound $\xi(\gamma_s^t) \geqslant - 2 \mu_0 \sqrt{K} t^{2/3}$ for $s \in [t_{k-1}, t_k]$ with $k \in \aA$ and also \eqref{eq:time_interval_length} for the length of the interval. The constant $c$ is independent of $\delta$ and $t$. 

We now turn to the second term on the right hand side of \eqref{eq:integral_early_split}. On the event $C_\eps(t),$ the Brownian motion $W$ is within the $\eps$-neighbourhood of the geodesic $\gamma^{t}$. In particular, we have
\begin{equation} \label{eq:CtyEst}
\left|\int_{0}^{\theta t} \xi(W_{s}) ds - \int_{0}^{\theta t} \xi(\gamma_{s}^{t}) ds \right|\leqslant \|\nabla\xi\|_{L^\infty(Q_{R(t)})} \, \eps \theta t\;.
\end{equation}
Since $\nabla \xi$ is again stationary Gaussian and has finite correlation length, by the same argument for Lemma~\ref{lem:XiGrowth}, one can show that $\|\nabla \xi\|_{L^\infty (Q_{R(t)})}$ is bounded by a constant multiple of $\sqrt{R(t)} \sim t^{2/3}$ for all large $t$ (depending on the law of $\xi$). Hence, there exists a deterministic constant $C>0$ such that for $\P$-a.e. realisation of $\xi$, we have
\begin{equation} \label{eq:integral_early_2}
\left| \int_{0}^{\theta t} \big( \xi(W_{s}) - \xi(\gamma_s^t) \big) ds \right| \leqslant C \eps t^{5/3}
\end{equation}
on the ($W$-)event $C_\eps(t)$ for all sufficiently large $t$. The proof of Proposition~\ref{prop:FormIntEst} is complete by combining the bounds \eqref{eq:integral_early_1} and \eqref{eq:integral_early_2}. 
\end{proof}

\subsection{Localisation around the geodesic before time $\theta t$ -- proof of Proposition~\ref{prop:probability_concentrating}}

We now turn to the proof of Proposition~\ref{prop:probability_concentrating}. Recall the partition \eqref{eq:time_interval_partition} for the time interval $[0, \theta t]$. For each $1 \leqslant k \leqslant N$, let $C_\eps^k(t)$ and $E^k(t)$ be the events
\begin{equation*}
    C_\eps^k(t) \triangleq \left\{ \sup_{s \in [t_{k-1}, t_k]} d(W_s, \gamma^t_s) < \eps \right\}\;, \qquad E^k(t) = \left\{ \; d(W_{t_k}, \gamma^t_{t_k}) < \zeta(t) \; \right\}\;.
\end{equation*}
All these events depend on $t$, but we will omit $t$ in the notation for simplicity. Hence, we have
\begin{equation*}
    \PP_o \big( C_\eps(t) \cap E(t) \big) \geqslant \PP_o \bigg( \bigcap_{k=1}^{N} \big( C_\eps^k \cap E^k \big) \bigg) = \EE_o \bigg[ \prod_{k=1}^{N} \1_{C_\eps^k \cap E^k} \bigg]\;.
\end{equation*}
Note that for each $k$, on the event $C_\eps^{k-1}$, we have
\begin{equation*}
    \PP^{W_{t_{k-1}}} \big( C_\eps^k \cap E^k \big) \geqslant \inf_{x \in B (\gamma^t_{t_{k-1}}, \; \zeta(t))} \PP \big( C_\eps^k \cap E^k \; | \; W_{t_{k-1}} = x \big) = \inf_{x \in Q_{\zeta(t)}} \PP_x \big( C_\eps^1 \cap E^1 \big)\;,
\end{equation*}
where the first inequality is valid on the event $C_\eps^{k-1}$, and the second equality follows from invariance of Brownian motion under action of the isometry group. Hence, by successively conditioning on $\fF_{t_{N-1}}, \dots, \fF_{t_1}$, we get
\begin{equation} \label{eq:prob_concentrating_factorisation}
    \PP_o \big( C_\eps(t) \cap E(t) \big) \geqslant \left( \inf_{x \in Q_{\zeta(t)}} \PP_x \big( C_\eps^1 \cap E^1 \big) \right)^N\;,
\end{equation}
where $N = \frac{R(t)}{\Lambda} = \frac{K t^{4/3}}{\Lambda} + \oO(t^{2/3})$. It then remains to estimate $\PP_x \big( C_\eps^1 \cap E^1 \big)$ over starting points $x \in Q_{\zeta(t)}$. We have
\begin{equation} \label{eq:prob_complement}
    \PP_x \big( C_\eps^1 \cap E^1 \big) = \PP_x (E^1) - \PP_x \big( (C_\eps^1)^c \cap E^1 \big)\;,
\end{equation}
and we control the two probabilities on the right hand side above from below and above respectively. 

\begin{lem} \label{lem:probability_entering}
    There exists $c>0$ depending on $\theta, K$ and $\Lambda$ such that
    \begin{equation*}
        t^{-7d/18} e^{-\frac{\Lambda K}{4 \theta} \cdot t^{1/3} - c t^{1/9}} \lesssim \PP_x \big( E^1 \big) \lesssim t^{-7d/18} e^{-\frac{\Lambda K}{4 \theta} \cdot t^{1/3} + c t^{1/9}}\;,
    \end{equation*}
    uniformly over all $x \in Q_{\zeta(t)}$ and all sufficiently large $t$. 
\end{lem}
\begin{proof}
    We have
    \begin{equation*}
        \PP_x \big( E^1 \big) = \int_{B(\gamma^t_{t_1}, \zeta(t))} p(s,x,y) \, \text{vol} (dy) \geqslant \text{vol} \big( Q_{\zeta(t)} \big) \cdot \inf_{x,y} p(s,x,y)\;,
    \end{equation*}
    where $s= t_1 = \frac{\theta \Lambda}{K} \cdot t^{-1/3} + \oO(t^{-1})$, and the infimum is taken over $x \in Q_{\zeta(t)}$ and $y \in B(\gamma^t_{t_1}, \zeta(t))$. Since $d(\gamma_{t_1}^t, o) = \Lambda$ by definition, we have
    \begin{equation*}
        \Lambda - 2 \zeta(t) < d(x,y) < \Lambda + 2 \zeta (t)
    \end{equation*}
    uniformly over $x,y$ in the above range. Hence, by the heat kernel asymptotics \eqref{eq:HKEst_small}, we have
    \begin{equation*}
        t^{-d/6} e^{-\frac{\Lambda K}{4 \theta} \cdot t^{1/3} - c t^{1/9}} \lesssim p(s,x,y) \lesssim t^{-d/6} e^{-\frac{\Lambda K}{4 \theta} \cdot t^{1/3} + c t^{1/9}}
    \end{equation*}
    for some $c>0$, uniformly over $x \in Q_{\zeta(t)}$, $y \in B(\gamma^t_{t_1}, \zeta(t))$ and sufficiently large $t$. Combining with the volume estimate $\text{vol}\big( Q_{\zeta(t)} \big) \asymp t^{-2d/9}$ (see \eqref{eq:volume_small} from Lemma~\ref{lem:volume}), we arrive at the conclusion of the lemma. 
\end{proof}

Now we consider the probability
\begin{equation*}
    \PP_x \big( (C_\eps^1)^c \cap E^1 \big) = \PP_x \left( \sup_{u \in [0,t_1]} d(W_u, \gamma_u^t) \geqslant \eps\;, \; \; d(W_{t_1}, \gamma^t_{t_1}) < \zeta(t) \right)\;. 
\end{equation*}
In order to apply the large deviations estimate for Brownian motion, we need to consider a slightly different event in which the Brownian motion $W$ and the geodesic have the same starting point. For this, let $\alpha^x: [0, t_1] \rightarrow \HH^d$ be the geodesic with $\alpha^x_0 = x$ and $\alpha^x_{t_1} = \gamma^t_{t_1}$. By convexity of geodesic property \eqref{eq:GeoConv}, we have
\begin{equation*}
    \sup_{u \in [0, t_1]} d(\alpha^x_u, \gamma^t_u) \leqslant d(x,o) < \zeta(t)\;.
\end{equation*}
Since $\zeta(t) \ll \eps$ for sufficiently small $t$, we have
\begin{equation*}
    \PP_x \big( (C_\eps^1)^c \cap E^1 \big) \leqslant \PP_x \left( \sup_{u \in [0, t_1]} d(W_u, \alpha^x_u) > \frac{\eps}{2}\;, \; \; d(W_{t_1}, \alpha^x_{t_1}) < \zeta(t) \right)\;.
\end{equation*}
The probability on the right hand side above depends on the starting point $x$ only via the quantity $d(x, \alpha_{t_1}^x) = d(x, \gamma_{t_1}^t)$, the length of the geodesic segment $\alpha^x$. Since the length is within a $\zeta(t)$ neighbourhood of $\Lambda$, we can think of length as fixed as $t \rightarrow +\infty$. 

\begin{lem} \label{lem:LDP}
    Let $\alpha: [0,s] \rightarrow \HH^d$ be a geodesic with length $L>0$. Then we have
    \begin{equation*}
        \PP_o \Big( \sup_{u \in [0,s]} d(W_u, \alpha_u) > \eps\;, \  d(W_s, \alpha_s) < \zeta \Big) \lesssim \exp \Big( -\frac{1}{4s} \big( L^2 + C \big( \eps^2 - \zeta \big) \big) \Big)
    \end{equation*}
    uniformly over all sufficiently small $\eps>0$, $\zeta \lesssim \eps^4$ and all sufficiently small $s$. 
\end{lem}

\begin{proof}
This is postponed to Appendix A.
\end{proof}

We are now ready to conclude the proof of Proposition~\ref{prop:probability_concentrating}. 

\begin{proof} [Proof of Proposition~\ref{prop:probability_concentrating}]
    Combining Lemmas~\ref{lem:probability_entering} and~\ref{lem:LDP}, and plugging them back to \eqref{eq:prob_complement} (with $s = \frac{\theta \Lambda}{K} \cdot t^{-1/3} + \oO(t^{-1})$ and $L \in (\Lambda - \zeta(t), \Lambda + \zeta(t))$), we have the lower bound
    \begin{equation*}
        \PP_x \big( C_\eps^1 \cap E^1 \big) \gtrsim e^{-\frac{\Lambda K}{4 \theta} \cdot t^{1/3}} \big( e^{-c_1 t^{1/9}} - e^{- c_2 t^{1/3}} \big) \gtrsim e^{-\frac{\Lambda K}{4 \theta} \cdot t^{1/3} - c t^{1/9}}
    \end{equation*}
    for all sufficiently large $t$ and uniformly over all $x \in Q_{\zeta(t)}$, where $c_1$ and $c_2$ are both independent of $t$ (though $c_2$ depends on $\eps$). Plugging this bound back to \eqref{eq:prob_concentrating_factorisation} and noting that $N = \frac{K t^{4/3}}{\Lambda} + \oO(t^{2/3})$, we obtain
    \begin{equation*}
        \PP_o \big( C_\eps(t) \cap E(t) \big) \geqslant C^{t^{4/3}} \cdot e^{-\frac{K^2}{4 \theta} \cdot t^{5/3} - c t^{13/9}}\;.
    \end{equation*}
    This concludes the proof. 
\end{proof}

\section{The matching upper asymptotics}
\label{sec:MainUp}

In this section, we prove the matching upper bound for \eqref{eq:MainLow}. The main theorem is stated as follows.

\begin{thm}
\label{thm:MainUp}Let $L^{*}$ be defined by (\ref{eq:OptimalExp}).
Then $\P$-almost surely, we have
\begin{equation} \label{eq:MainUp}
\underset{t\rightarrow\infty}{\overline{\lim}}\frac{1}{t^{5/3}}\log u(t,o)\leqslant L^{*}.
\end{equation}
\end{thm}

In the proof for the lower bound, we identified a particular scenario
of Brownian motion under which the solution $u(t,o)$ picks up the
growth of $e^{L^{*}t^{4/3}}$. To prove the matching upper bound \eqref{eq:MainUp}, we have to take into account all possible Brownian trajectories and show that \textit{no Brownian scenarios (nor their total contribution) could produce a growth exponent larger than} $L^{*}.$ This is more challenging than the lower bound. In the following subsections, we develop the major steps for the proof of Theorem~\ref{thm:MainUp} in a mathematically precise way.

\subsection{An initial localisation}

We first localisie the Brownian motion on a large fixed ball (with radius of size $\oO(t^{4/3})$). 

\begin{lem}\label{lem:RoughLocal}
There exist a deterministic constant $K_{0}$ such that for almost every realisation of $\xi$, we have 
\[
\mathbb{E}_o \Big[e^{\int_{0}^{t}\xi(W_{s})ds}; \sup_{s \in [0,t]} d(W_{s},o) > K_{0}t^{4/3} \Big] \leqslant e^{- t^{5/3}}
\]
for all large $t$ (depending on the realisation of $\xi$). 
\end{lem}
\begin{proof}
Fix $K_0>0$ whose value will be specified later. For every $t \geqslant 1$, let $\nN_n = \nN_n(t)$ be the event
\begin{equation*}
    \nN_n \triangleq \Big\{\sup_{s \in [0,t]} d(W_{s},o) \in \big(n K_0 t^{4/3},(n+1) K_0 t^{4/3} \big] \Big\}\;.
\end{equation*}
The events $\{\nN_n\}_n$ depend on $t$ and $K_0$ but we omit them for notational simplicity. We then have
\begin{equation} \label{eq:Up4}
    \mathbb{E}_o \Big[e^{\int_{0}^{t}\xi(W_{s})ds}; \sup_{s \in [0,t]} d(W_{s},o) > K_0 t^{4/3} \Big] = \sum_{n=1}^{\infty} \mathbb{E}_o \big[e^{\int_{0}^{t}\xi(W_{s})ds}; \, {\cal N}_{n}\big]\;.
\end{equation}
By Lemma~\ref{lem:XiGrowth}, for almost every realisation of $\xi$, there exists $T>0$ such that on each $W$-event $\nN_n$, we have
\begin{equation*}
    \sup_{s \in [0,t]} \xi(W_{s}) \leqslant 2 \mu_{0} \sqrt{(n+1) K_0 t^{4/3}}
\end{equation*}
for all $t>T$. Note that this (random) $T$ depends on the realisation of $\xi$ only, and is independent of $n$ and $K_0$. It then follows from \eqref{eq:ExitEst} that
\begin{equation} \label{eq:Upper2}
\begin{split}
\mathbb{E}_o \big[e^{\int_{0}^{t}\xi(W_{s})ds}; {\cal N}_{n}\big] &\leqslant e^{2\mu_{0} \sqrt{(n+1) K_0} \, t^{5/3}} \mathbb{P}_o \Big(\sup_{s \in [0,t]}d(W_{s},o) \geqslant n K_0 \, t^{4/3} \Big)\nonumber \\
&\leqslant C e^{\big(2\mu_{0}\sqrt{(n+1) K_0} - c  n^{2} K_0^{2}\big)t^{5/3}}
\end{split}
\end{equation}
for some $c,C>0$. By choosing $K_{0}$ sufficiently large, one can ensure that 
\[
2\mu_{0}\sqrt{(n+1) K_{0}} - c n^{2} K_{0}^{2} < - 2 n
\]
for all $n \geqslant 1$. For this choice of $K_{0}$, the sum on the right hand side of \eqref{eq:Up4} is bounded by $e^{-t^{5/3}}$ for all sufficiently large $t$. This proves the lemma. 
\end{proof}

From now on, we fix the parameter $K_{0}$ as in Lemma \ref{lem:RoughLocal}.
Our task thus reduces to proving the upper bound (\ref{eq:MainUp})
with the unconditional expectation replaced by 
\begin{equation}
\mathbb{E}\big[e^{\int_{0}^{t}\xi(W_{s})ds};M^{t}\big],\label{eq:ExpLocal}
\end{equation}
where $M^{t}\triangleq\{d(W_{s},o)\leqslant K_{0}t^{4/3}\ \forall s\in[0,t]\}.$

\subsection{Islands, clusters and cluster sequences}

We introduce the notion of islands and clusters that depend on the field $\xi$, and prove a few properties that will be important to us later on. In what follows, we fix two parameters $\delta, \lambda > 0$ (we always assume $\lambda < \delta^8$). Let $t>0$ be given. The set
\begin{equation} \label{eq:set_large}
    \mM_\delta^t \triangleq \left\{ x \in Q_{K_0 t^{4/3}}: \; \xi(x) > \delta \, t^{2/3}  \right\}
\end{equation}
is almost surely a disjoint union of finitely many open sets.

\begin{defn} [Islands and clusters]
\label{def:cluster}
Recall $t>0$ and $\delta, \lambda \in (0,1)$ with $\lambda < \delta^8$ are fixed parameters, and the set $\mM_\delta^t$ from \eqref{eq:set_large}. 
\begin{enumerate} [(i)]
    \item A $(\delta,t)$-island is a connected component of $\mM_\delta^t$. The collection of $(\delta,t)$-islands is denoted by $\iI_\delta^t$. 
    
    \item Two points $x, y \in \mM_\delta^t$ are called $(\lambda,t)$-connected (in $\mM_\delta^t$), denoted by $\sim_{(\lambda,t)}$, if there exists $N \in \NN$ and points $x_1, \dots, x_N \in \mM_\delta^t$ such that
    \begin{equation*}
        d(x_j, x_{j+1}) < \lambda \, t^{4/3}
    \end{equation*}
    for all $j=0,1,\dots, N$, where we denote $x_0 = x$ and $x_{N+1} = y$. The relation $\sim_{(\lambda,t)}$ is an equivalence relation. 

    \item The relation $\sim_{(\lambda,t)}$ partitions points in $\mM_\delta^t$ into equivalence classes. Write
    \begin{equation*}
        \Clus_{\delta,\lambda}^t \triangleq \Big\{ \text{Equivalence classes of elements in} \; \mM_\delta^t \; \text{under} \; \sim_{(\lambda,t)} \Big\}\;.
    \end{equation*}
    An element in $\Clus_{\delta,\lambda}^t$ is called a $(\lambda,t)$-cluster of $(\delta,t)$-islands. 
\end{enumerate}
\end{defn}

By definition, an element (cluster) in $\Clus_{\delta,\lambda}^t$ is a finite union of $(\delta,t)$-islands. Different elements (clusters) in $\Clus_{\delta,\lambda}^t$ are at least $\lambda t^{4/3}$-apart from each other. Note that the islands and clusters are random sets since they are defined for each realisation of $\xi$. The following lemma summarises the essential properties of $(\lambda,t)$-clusters that are needed for our purpose. Recall that $r_0$ is the correlation length of the Gaussian field $\xi$, as well as the constant $C_r$ from Lemma~\ref{lem:XiTail}.

\begin{lem} \label{lem:ClusterProperty}
For $\delta>0$, define
\begin{equation} \label{eq:LCond}
L_{\delta} \triangleq \delta^{-4}\;.
\end{equation}
Then there exists a deterministic $\delta_0>0$ such that for $\P$-a.e. realisation of $\xi$, there exists (a random) $T>0$ such that the followings are true for all $t>T$, $\delta < \delta_0$ and $\lambda < \delta^8$: 
\begin{enumerate}
\item Every cluster $\fC \in \Clus_{\delta,\lambda}^t$ has diameter at most $L_\delta \lambda t^{4/3}$. 

\item Every cluster $\fC \in \Clus_{\delta,\lambda}^t$ contains at most $L_\delta$ points in $\mM_{\delta}^{t}$ that are at least $9r_0$-apart from each other.
\end{enumerate}
\end{lem}
\begin{proof}
We first prove Assertion 1. Let $\fC \in \Clus_{\delta,\lambda}^t$ with $\diam(\fC) = \eta t^{4/3}$. Then, there exist $x,y \in \fC$ with $d(x,y) > \frac{\eta t^{4/3}}{2}$. For each $k \geqslant 1$, let
\begin{equation*}
    A_{k} \triangleq \left\{z\in Q_{K_{0}t^{4/3}}:(k-1) \lambda t^{4/3} < d(z,x) \leqslant k \lambda t^{4/3} \right\}\;.
\end{equation*}
By Definition~\ref{def:cluster}, there exists $x_k \in A_k \cap \fC$ for each $k=1, \dots, \frac{\eta}{2 \lambda}$, for otherwise $x,y$ cannot belong to the same cluster in $\Clus_{\delta,\lambda}^t$. Thus, the points $\{x_{2k}\}_{k=1}^{\eta / 4\lambda}$ are in $\fC$ and are at least $\lambda t^{4/3}$-distance away from each other. 

If $\eta > L_\delta \lambda$, then the first $L_\delta$ points $x_2, x_4, \dots, x_{2 L_\delta}$ of the above sequence are all contained in the ball centered at $x$ with radius $2 L_\delta \lambda t^{4/3} < 2 \delta^4 t^{4/3}$ if $\lambda < \delta^8$. This contradicts Proposition~\ref{prop:local_maxima_sparse} for sufficiently large $t$ (since $2 \delta^4 < \delta^3$ for $\delta$ small, and $\lambda t^{4/3} > 9 r_0$ for large $t$). Hence, we necessarily have $\eta < L_\delta \lambda$. This proves Assertion 1. 

Since any $\fC \in \Clus_{\delta,\lambda}^t$ has diameter at most $L_\delta \lambda t^{4/3} < \delta^4 t^{4/3}$ and is thus contained in a ball of the same radius, Assertion 2 then follows from Proposition~\ref{prop:local_maxima_sparse}. 
\end{proof}

Before we proceed, we introduce one more notation. Let
\begin{equation*}
    \big( \Clus_{\delta,\lambda}^t \big)^{\NN_0} \triangleq \left\{ \vec{\fS} = (\fC_1, \, \dots, \, \fC_m)\,: \; m \in \NN^+\,, \; \fC_k \in \Clus_{\delta,\lambda}^t \right\}
\end{equation*}
be the collection of all ordered finite sequences of elements in $\Clus_{\delta,\lambda}^t$. Suppose $\vec{\fS} \in (\Clus_{\delta,\lambda}^t)^{\NN_0}$ has the form $\vec{\fS} = (\fC_1, \, \dots, \, \fC_m)$. Write $|\vec{\fS}| = m$ for the length of $\vec{\fS}$, and $\vec{\fS}_k = \fC_k \in \Clus_{\delta,\lambda}^t$ be the $k$-th element in this cluster sequence.

\begin{rmk}
    The reason we set $L_\delta = \delta^{-4}$ instead of $\delta^{-3}$ as in Proposition~\ref{prop:local_maxima_sparse} is that the setups are slightly different in the multiplicative constants. Hence we add another power of $\delta$ to kill the effect of multiplicative constants. 
\end{rmk}


%


\subsection{Discrete Brownian routes and a further localisation}

To bound the expectation \eqref{eq:ExpLocal} from above for each fixed realisation of $\xi$, we need to carefully analyse the contributions from \textit{all possible ways} of the Brownian motion propagating through islands and clusters. For each continuous path $\gamma: [0,t] \rightarrow \HH^d$, we associate it with a unique element in $\tT_{\delta,\lambda}^t(\gamma) \in \big( \Clus_{\delta,\lambda}^t \big)^{\NN_0}$ as follows.

\begin{defn} [Discrete trajectory]
\label{def:DiscRt}
Fix $\gamma: [0,t] \rightarrow \HH^d$ continuous. Let $\tau_0 \triangleq 0$. For each $k \geqslant 1$, define the stopping times $\sigma_k$ and $\tau_k$ by
\begin{equation*}
    \sigma_{k} \triangleq \inf \big\{ s > \tau_{k-1}: \; \gamma_s \in \Clus_{\delta,\lambda}^t \big\}\;, \quad \tau_k \triangleq \inf \Big\{ s > \sigma_k: \; d(\gamma_s, \Clus_{\delta,\lambda}^t) > \frac{\lambda t^{4/3}}{2} \Big\}\;.
\end{equation*}
Let $\fC_k \in \Clus_{\delta,\lambda}^t$ be such that $\gamma_{\sigma_k} \in \overline{\fC_k}$. Let $m$ be the unique integer such that $\sigma_m < t$ and $\sigma_{m+1} \geqslant t$. Then the sequence of ordered clusters
\begin{equation*}
    \tT_{\delta,\lambda}^t(\gamma) \triangleq (\fC_1, \, \fC_2, \, \cdots, \, \fC_m) \in \big( \Clus_{\delta,\lambda}^t \big)^{\NN_0}
\end{equation*}
is called the (discrete) trajectory of $\gamma$ in the clusters $\Clus_{\delta,\lambda}^t$. 
\end{defn}





Before studying a generic Brownian scenario, we perform one more localisation on the Brownian trajectory.

\begin{prop} \label{prop:ongRtEst}
Let $N\geqslant1$ be given fixed. Then there exists a constant $C = C(\lambda,N)$ such that for $\P$-a.e. realisation of $\xi$, we have
\begin{equation*}
\EE_o \Big[ e^{\int_{0}^{t}\xi(W_{s})ds}; \, M^{t} \cap \left\{ |\tT_{\delta,\lambda}^t|> N \right\} \Big] \leqslant C(\lambda,N) \cdot \exp \Big[ \Big(2\mu_{0}\sqrt{K_{0}} -\frac{(\lambda N)^{2}}{72} \Big)t^{5/3} \Big]
\end{equation*}
for all sufficiently large $t$. 
\end{prop}

The main ingredient is the following induction estimate. 

\begin{lem} \label{lem:conditional_induction}
    Let $\Phi: [0,t] \rightarrow \RR$ be a non-decreasing function. Let $\{\sigma_k\}_{k \geqslant 1}$ be the stopping times with respect to $\Clus_{\delta,\lambda}^t$ as given in Definition~\ref{def:DiscRt}. Then we have
    \begin{equation*}
        \EE \Big[ \1_{\{\sigma_k < t\}} \, \Phi(t-\sigma_k) \Big] \leqslant \EE \Big[ \1_{\{\sigma_{k-1}<t\}} \, \big(\Phi * (\1_{[0, t-\sigma_{k-1}]} \cdot q_{\lambda,t}) \big)(t-\sigma
        _{k-1}) \Big]\;,
    \end{equation*}
    where
    \begin{equation} \label{eq:q_lambda}
        q_{\lambda,t}(u) = \frac{\lambda t^{4/3}}{6 \sqrt{\pi}} \, \1_{u>0} \, u^{-3/2} \, e^{-\frac{\lambda^2 t^{8/3}}{36 u}}\;.
    \end{equation}
\end{lem}
\begin{proof}
    By conditioning on $\fF_{\sigma_{k-1}}$, we have
    \begin{equation} \label{eq:conditiona_induction}
        \EE \Big[ \1_{\{\sigma_k < t\}} \, \Phi(t-\sigma_k) \Big] = \EE \Big[ \1_{\{\sigma_{k-1} < t\}} \, \EE \big( \1_{\{\sigma_k < t\}} \, \Phi(t-\sigma_k) | \fF_{\sigma_{k-1}} \big) \Big]\;.
    \end{equation}
    It suffices to show that
    \begin{equation} \label{eq:conditional_step_bound}
        \EE \big( \1_{\{\sigma_k < t\}} \, \phi(t-\sigma_k) | \fF_{\sigma_{k-1}} \big) \leqslant (\Phi * q_{\lambda,t})(t-\sigma_{k-1})\;.
    \end{equation}
    Note that the left hand side above equals
    \begin{equation*}
        \EE^{W_{\sigma_{k-1}}} \Big( \1_{\{\sigma_k - \sigma_{k-1} < t - \sigma_{k-1} \}} \cdot \Phi \big( (t - \sigma_{k-1}) - (\sigma_k - \sigma_{k-1}) \big) \Big)\;.
    \end{equation*}
    By definition, $\sigma_k$ is strictly larger than the first time after $\sigma_{k-1}$ that the Brownian motion $W$ is at least $\lambda t^{4/3}/2$ away from $W_{\sigma_{k-1}}$. Also note that the function $\Phi \big( (t-\sigma_{k-1}) - (\sigma_k - \sigma_{k-1}) \big)$ is \textit{non-increasing} in the variable $\sigma_k - \sigma_{k-1}$. Hence, by the strong Markov property, the desired bound \eqref{eq:conditional_step_bound} follows from a direct application of Proposition~\ref{prop:hitting_time} (with $s=t - \sigma_{k-1}$, $\Theta = \lambda t^{4/3}/2$ and $\alpha = 2/3$). The claim of the lemma then follows by plugging the bound \eqref{eq:conditional_step_bound} back into \eqref{eq:conditiona_induction}. 
\end{proof}

We are now ready to prove Proposition~\ref{prop:ongRtEst}.

\begin{proof}[Proof of Proposition~\ref{prop:ongRtEst}]
By \eqref{eq:XiUp}, on the event $M^t$, $\xi(W_s)$ is bounded by $2 \mu_0 \sqrt{K_0} t^{2/3}$ for all $s \in [0,t]$. Hence, we have
\begin{equation} \label{eq:LRXiUp}
\mathbb{E}\big[e^{\int_{0}^{t}\xi(W_{s})ds}; \, M^{t} \cap\{|\tT_{\delta,\lambda}^t|>N\} \big] \leqslant e^{2 \mu_0 \sqrt{K_0} t^{5/3}} \cdot \PP \big( |\tT_{\delta,\lambda}^t| > N \big)
\end{equation}
for all sufficiently large $t$. It remains to estimate the
probability $\PP ( |\tT_{\delta,\lambda}^t| > N )$.

Recall that $\sigma_{1},\sigma_{2},\dots$ denote the successive
visit times of the $\lambda$-clusters given in Definition~\ref{def:DiscRt}. Applying Lemma~\ref{lem:conditional_induction} once, we get
\begin{equation*}
    \PP \big( |\tT_{\delta,\lambda}^t| > N \big) = \EE \big( \1_{\{\sigma_N < t\}} \big) \leqslant \EE \Big[ \1_{\{\sigma_{N-1}<t\}} \, \Phi^{(1)} (t-\sigma_{N-1}) \Big]\;,
\end{equation*}
where
\begin{equation*}
    \Phi^{(1)}(t-\sigma_{N-1}) = \int_{0}^{t-\sigma_{N-1}} q_{\lambda,t}(u) du\;,
\end{equation*}
and $q_{\lambda,t}$ is the function given in \eqref{eq:q_lambda}. Note that the integral on the right hand side above is an increasing function in the variable $t-\sigma_{N-1}$, which allows another application of Lemma~\ref{lem:conditional_induction}. 

Each time one applies Lemma~\ref{lem:conditional_induction}, one updates the function $\Phi^{(k)}$ to $\Phi^{(k+1)}$, which is obtained from $\Phi^{(k)}$ by convolution with $q_{\lambda,t}$, and thus giving another increasing function. Hence, by repeatedly applying Lemma~\ref{lem:conditional_induction} $N$ times, we get
\begin{equation*}
    \PP \big( |\tT_{\delta,\lambda}^t| > N \big) \leqslant \int_{0 < u_1 + \cdots u_N < t} \; \Big( \prod_{j=1}^{N} q_{\lambda,t}(u_j) \Big) \; d u_1 \cdots d u_N\;.
\end{equation*}
By a change of variable $v_j \triangleq u_j / (\lambda^2 t^{8/3})$, one obtains
\begin{equation} \label{eq:LRPf2}
    \PP \big( |\tT_{\delta,\lambda}^t| > N \big) \leqslant (6 \sqrt{\pi})^{-N} \int_{0 < \sum_j v_j < \frac{1}{\lambda^2 t^{5/3}}} \Big( \prod_{j=1}^{N} v_j^{-3/2} \Big) \cdot \exp \bigg( -\frac{1}{36} \sum_{j=1}^{N} \frac{1}{v_j} \bigg) d \mathbf{v}\;,
\end{equation}
where we used the notation $d \mathbf{v} = d v_1 \dots d v_N$. By the harmonic-arithmetic mean inequality and the range of integration, we have
\begin{equation*}
\frac{N}{1/v_{1}+\cdots+1/v_{N}} \leqslant \frac{v_{1}+\cdots+v_{N}}{N}\leqslant\frac{1}{\lambda^2 t^{5/3} N}\;,
\end{equation*}
which in turn implies that the quantity on the exponential satisfies
\begin{equation*}
    -\frac{1}{36} \sum_{j=1}^{N} \frac{1}{v_j} = -\frac{1}{72} \sum_{j=1}^{N} \frac{1}{v_j} - \frac{1}{72} \sum_{j=1}^{N} \frac{1}{v_j} \leqslant -\frac{1}{72} \sum_{j=1}^{N} \frac{1}{v_j} - \frac{(\lambda N)^2}{72} \cdot t^{5/3}\;.
\end{equation*}
Plugging it back to \eqref{eq:LRPf2} and enlarging the domain of integration to $[0,t]^N$, we get
\begin{equation} \label{eq:LRPf3}
\PP \big( |\tT_{\delta,\lambda}^t| > N \big) \leqslant \exp \bigg( -\frac{(\lambda N)^2}{72} \cdot t^{5/3} \bigg) \cdot \bigg( \frac{1}{6 \sqrt{\pi}} \int_{0}^{1/(\lambda^2 t^{5/3})} v^{-\frac{3}{2}} \cdot e^{-\frac{1}{72 v}} dv \bigg)^N\;.
\end{equation}
The integral term on the right hand side above is uniformly bounded by $1$ for all $\lambda$ and $t$ with $\lambda^2 t^{5/3} \gg 1$. The conclusion of the proposition then follows by combining \eqref{eq:LRXiUp} and \eqref{eq:LRPf3}. 
\end{proof}

\begin{rem}
In the above proof, technically the visit of the first cluster $\fC_1$ is different from the rest, since one does not know if $d(o,\fC_1) > \lambda t^{4/3}$. However, since $\xi(x)>\delta t^{2/3}$ for any $x \in \fC_1$, one knows from \eqref{eq:XiUp} that $d(o, \fC_1)>(\delta/2\mu_{0})^{2}t^{4/3}$. To make the argument for the first visit consistent with the rest, one could further impose that $\lambda<(\delta/2\mu_{0})^{2}$.  
\end{rem}

\subsection{The main estimate on generic routes}
\label{sec:main_estimate_upper}

In this subsection, we estimate the contribution from a generic discrete route. In what follows, let
\begin{equation*}
    \eta = \eta (\delta) = \delta^{20}\;, \qquad \eta' = \eta' (\delta) = \delta^8\;.
\end{equation*}
By Proposition~\ref{prop:ongRtEst}, there exists $C_0>0$ depending on $\mu_0$ and $K_0$ only such that
\begin{equation} \label{eq:full_localisation}
    \E \Big[ e^{\int_{0}^{t} \xi(W_s) ds}; \; (M^t)^c \cup \big\{ |\tT_{\delta,\eta'}^t| > C_0 / \eta' \big\} \Big] \leqslant e^{- c t^{5/3}}
\end{equation}
for all sufficiently large $t$. Hence, we may restrict ourselves to the intersection of the events $M^t \cap \{ |\tT_{\delta,\eta'}^t| \leqslant C_0/\eta' \}$, and consider the quantity
\begin{equation*}
    \begin{split}
    &\phantom{111}\E \Big[ e^{\int_{0}^{t} \xi(W_s) ds}; \; M^t \cap \big\{ |\tT_{\delta,\eta'}^t| \leqslant C_0/\eta' \big\}  \Big]\\
    &= \sum_{\vec{\fS} \in (\Clus_{\delta,\eta}^t)^{\NN_0}} \E \Big[ e^{\int_{0}^{t} \xi(W_s) ds}; \; M^t \cap \big\{ |\tT_{\delta,\eta'}^t| \leqslant C_0/\eta' \big\} \cap \big\{ \tT_{\delta,\eta}^t = \vec{\fS} \big\} \Big]\;.
    \end{split}
\end{equation*}
Note that the sum is taken over all ordered finite sequences in $\Clus_{\delta,\eta}^t$ (and not $\Clus_{\delta,\eta'}^t$). 

It turns out that the other restriction $\{|\tT_{\delta,\eta'}^t| \leqslant C_0 / \eta'\}$ puts further constraints on the above sum over $\vec{\fS}$. To better characterise this further constraint, We first introduce the notion of a reduced $\Clus_{\delta,\eta}^t$ sequence. 

\begin{defn} [Reduction operation]
\label{def:reduced_cluster}
    Let $\vec{\fS} \in (\Clus_{\delta,\eta}^t)^{\NN_0}$, and $m^* \leqslant |\vec{\fS}|$ be such that
    \begin{equation} \label{eq:furthest_cluster}
        \sup_{x \in \vec{\fS}_{m^*}} d(x,o) = \sup_{1 \leqslant k \leqslant |\vec{\fS}|} \; \sup_{x \in \vec{\fS}_k} d(x,o)\;.
    \end{equation}
    Let $j_0 = 0$. For each $\ell$, let
    \begin{equation*}
        j_{\ell+1} \triangleq \sup \big\{ k \leqslant m^*\,: \; \vec{\fS}_k = \vec{\fS}_{j_{\ell}+1}  \big\}\;.
    \end{equation*}
    In other words, $j_1$ is the largest number before $m^*$ with $\vec{\fS}_{j_1} = \vec{\fS}_1$, $j_2$ is the largest number before $m^*$ with $\vec{\fS}_{j_2} = \vec{\fS}_{j_1 + 1}$, etc. This procedure necessarily ends with finitely many steps, say $\bar{m}$. We call
    \begin{equation*}
        \rR({\vec{\fS}}) \triangleq (\vec{\fS}_{j_1}, \, \dots, \, \vec{\fS}_{j_{\bar{m}}})
    \end{equation*}
    the \textit{reduced sequence} of $\vec{\fS}$ in $(\Clus_{\delta,\lambda}^t)^{\NN_0}$. 
\end{defn}

\begin{rmk}
    The reduction operation is not uniquely defined since the choice of $m^*$ satisfying \eqref{eq:furthest_cluster} may not be unique (but it is $\P$-a.s. unique). On the other hand, given $m^*$, the operation $\rR$ is then uniquely defined. 
\end{rmk}

\begin{lem} \label{lem:reduced_cluster_length}
    There exists $C>0$ such that
    \begin{equation*}
        |\rR \big(\tT_{\delta,\eta}^t(\gamma) \big)| \leqslant C \delta^{-12}
    \end{equation*}
    uniformly over all paths $\gamma: [0,t] \rightarrow \HH^d$ with the restriction $|\tT_{\delta,\eta'}^t(\gamma)| \leqslant C_0 / \eta'$. 
\end{lem}
\begin{proof}
    Suppose $\tT_{\delta,\eta}^t = \vec{\fS}$, and let $\rR(\vec{\fS})$ be a reduction of $\vec{\fS}$ in $(\Clus_{\delta,\eta}^t)^{\NN_0}$ as in Definition~\ref{def:reduced_cluster}. By construction, $\rR(\vec{\fS})_k$ are different for all $k \leqslant |\rR(\vec{\fS})|$. Hence, the number of different clusters in $\Clus_{\delta,\eta}^t$ that $\gamma$ passes through is at least $|\rR(\vec{\fS})|$. 

    On the other hand, by Lemma~\ref{lem:ClusterProperty}, each cluster in $\Clus_{\delta,\eta'}^t$ contains at most $L_\delta = \delta^{-4}$ different (smaller) clusters in $\Clus_{\delta,\eta}^t$. Since $|\tT_{\delta,\eta'}^t| \leqslant C_0 / \eta'$, the number of different clusters that $\gamma$ passes through can at most be $L_\delta \, |\tT_{\delta,\eta'}^t| \lesssim \delta^{-12}$. Hence, we necessarily have $|\rR(\vec{\fS})| \leqslant C  \delta^{-12}$. This completes the proof. 
\end{proof}

Let $\gG_{\delta,\eta,\eta'}^t \subset (\Clus_{\delta,\eta}^t)^{\NN_0}$ be the set
\begin{equation*}
    \gG_{\delta,\eta,\eta'}^t = \bigg\{ \vec{\fS} \in (\Clus_{\delta,\eta}^t)^{\NN_0}\,: \; \bigcup_{k=1}^{|\vec{\fS}|} \vec{\fS}_k \subset Q_{K_0 t^{4/3}} \; \text{and} \; |\rR(\vec{\fS})| < \delta^{-13} \bigg\}\;.
\end{equation*}
By Lemma~\ref{lem:reduced_cluster_length}, we have
\begin{equation*}
    \E \Big[ e^{\int_{0}^{t} \xi(W_s) ds}; \; M^t \cap \big\{ |\tT_{\delta,\eta'}^t| \leqslant C_0/\eta' \big\} \Big] \leqslant \sum_{\vec{\fS} \in \gG_{\delta,\eta,\eta'}^t} \E \Big[ e^{\int_{0}^{t} \xi(W_s) ds} \; \1_{\{\tT_{\delta,\eta}^t = \vec{\fS}\}} \Big]\;.
\end{equation*}
The mains statement in this subsection is the following proposition. 

\begin{prop} \label{prop:KeyEst}
For every $\alpha \in (0,1)$, $\delta>0$ and $\mu > \mu_0$, there exists $C = C(\delta,\eta,\eta',\alpha,\mu)$ such that
\begin{equation} \label{eq:KeyEst}
    \EE \Big[ e^{\int_{0}^{t} \xi(W_s) ds} \, \1_{\{\tT_{\delta,\eta}^t = \vec{\fS}\}} \Big] \leqslant C^{|\vec{\fS}|} \exp \Big( \big( L^* + \delta + \mu \sqrt{K_0} - \alpha \mu_0 \sqrt{K_0 - \delta^2} \big) t^{5/3} \Big)
\end{equation}
for all large $t$ uniformly in sequences $\vec{\fS}$ in $\gG_{\delta,\eta,\eta'}^t$.
\end{prop}

In what follows, we develop the main steps towards proving Proposition
\ref{prop:KeyEst}.

\subsubsection{Decomposing $\int_{0}^{t}\xi(W_{s})ds$ into ``excursion'' and ``staying'' parts}

Suppose $\vec{\mathfrak{S}} \in (\Clus_{\delta,\eta}^t)^{\NN_0}$ with $|\vec{\fS}| = m$. For this cluster sequence $\vec{\fS}$, let $\tau_0 \triangleq 0$. For each $1 \leqslant k \leqslant m$, define the stopping times $\sigma_k = \sigma_k(\vec{\fS})$ and $\tau_k = \tau_k(\vec{\fS})$ recursively by
\begin{equation} \label{eq:stopping_time_cluster_specific}
    \begin{split}
    \sigma_k &\triangleq t \wedge \inf \big\{ s > \tau_{k-1}: \; W_s \in \vec{\fS}_k \big\}\;,\\
    \tau_k &\triangleq t \wedge \inf \Big\{ s > \sigma_k: \; d(W_s, \vec{\fS}_k) > \frac{\eta \, t^{4/3}}{2} \Big\}\;, \qquad \tau_m \triangleq t\;.
    \end{split}
\end{equation}
Note that we have made an abuse of notation with Definition~\ref{def:DiscRt} since the stopping times here are associated with the fixed sequence $\vec{\fS}$. Let
\begin{equation} \label{eq:K*Def}
K \triangleq t^{-4/3} \, \sup \Big\{d(x,o): x \in \bigcup_{k=1}^{m} \vec{\fS}_k \Big\}\;.
\end{equation}
Note that $K = K_t$ also depends on $t$, but by Lemma~\ref{lem:XiGrowth}, it satisfies
\begin{equation*}
    \Big( \frac{\delta}{\mu_0} \Big)^2 \leqslant K \leqslant K_0
\end{equation*}
for all sufficiently large $t$. By definition of $\Clus_{\delta,\eta}^t$ and the stopping times in \eqref{eq:stopping_time_cluster_specific} as well as Lemma~\ref{lem:XiGrowth}, we have
\begin{equation*}
    \xi(W_s) \leqslant
    \begin{cases}
\delta t^{2/3}\;, & s \in [\tau_{k-1},\sigma_k]\\
\mu \sqrt{K} t^{2/3}\;, & s \in [\sigma_{k-1}, \tau_k]
\end{cases}
\qquad \text{on} \; \big\{ \tT_{\delta,\eta}^t = \vec{\fS} \big\}
\end{equation*}
for every $\mu > \mu_0$ and all sufficiently large $t$. As a consequence, we deduce
\begin{equation*}
    \begin{split}
    \int_{0}^{t} \xi(W_s) ds &= \sum_{k=1}^{m} \int_{\tau_{k-1}}^{\sigma_k} \xi(W_s) ds + \sum_{k=1}^{m} \int_{\sigma_{k-1}}^{\tau_k} \xi(W_s) ds\\
    &\leqslant \delta t^{5/3} + \mu \sqrt{K} t^{2/3} \sum_{k=1}^{m-1} (\tau_k - \sigma_{k-1}) \qquad \text{on} \; \{\tT_{\delta,\eta}^t = \vec{\fS}\}\;,
    \end{split}
\end{equation*}
for every $\mu > \mu_0$, $\delta < \delta_0$ and all sufficiently large $t$. Exponentiate the above and restricting to the event $\{\tT_{\delta,\eta}^t = \vec{\fS}\}$, we get the bound
\begin{equation} \label{eq:KeyEstPf1}
    \EE \Big[ e^{\int_{0}^{t} \xi(W_s) ds} \, \1_{\{\tT_{\delta,\eta}^t = \vec{\fS}\}} \Big] \leqslant e^{\delta t^{5/3}} \EE \bigg[ \exp \Big( \mu \sqrt{K} t^{2/3} \sum_{k=1}^{m} \big( \tau_k - \sigma_{k-1} \big) \Big) \, \1_{\{\tT_{\delta,\eta}^t = \vec{\fS}\}} \bigg]
\end{equation}
for every $\mu > \mu_0$ and all sufficiently large $t$.

\subsubsection{Estimating the ``staying'' part}

Now our aim is to estimate the expectation on the right hand side of \eqref{eq:KeyEstPf1} above. For each $k=1, \dots, m$ and $\alpha \in (0,1)$, let
\begin{equation*}
    D_k \triangleq \text{dist} \big( \vec{\fS}_{k-1}, \vec{\fS}_k \big) - \frac{\eta \, t^{4/3}}{2}\;, \qquad q_k^\alpha(u) = \frac{\alpha D_k}{2 \sqrt{\pi}} \, u^{-3/2} \, e^{-\frac{\alpha^2 D_k^2}{4u}}\;,
\end{equation*}
Let
\begin{equation*}
    \widehat{\sigma}_k \triangleq \inf \big\{ s>0\,: \; d(W_{\tau_{k-1}+s}, \; W_{\tau_{k-1}}) > D_k \big\}\;.
\end{equation*}
Then we have $\widehat{\sigma}_k \leqslant \sigma_k - \tau_{k-1}$ for every $k$. To estimate the ``jumping times" from $\partial \widehat{\fC_{k-1}}$ to $\fC_k$, we rewrite
\begin{equation} \label{eq:staying_to_jumping}
    \sum_{k=1}^{m} \big( \tau_k - \sigma_{k-1} \big) = t - \sum_{k=1}^{m} \big( \sigma_k - \tau_{k-1} \big) \leqslant t - \sum_{k=1}^{m} \widehat{\sigma}_k\;.
\end{equation}
Also, note that
\begin{equation} \label{eq:indicator_expansion}
    \1_{\{\tT_{\delta,\eta}^t = \vec{\fS}\}} = \prod_{k=1}^{m} \1_{\{\sigma_k < t\}} \leqslant \prod_{k=1}^{m} \1_{\{\widehat{\sigma}_k < t - \tau_{k-1}\}}\;.
\end{equation}
Plugging \eqref{eq:staying_to_jumping} and \eqref{eq:indicator_expansion} back into \eqref{eq:KeyEstPf1}, we have
\begin{equation} \label{eq:KeyEstPf2}
    \EE \Big[ e^{\int_{0}^{t} \xi(W_s) ds} \, \1_{\{\tT_{\delta,\eta}^t = \vec{\fS}\}} \Big] \leqslant e^{(\delta + \mu \sqrt{K}) t^{5/3}} \, \EE \Big[ \prod_{k=1}^{m} \Big( \phi(\widehat{\sigma}_k) \, \1_{\{ \widehat{\sigma}_k < t - \tau_{k-1} \}} \Big) \Big]\;,
\end{equation}
where $\phi(\sigma) = e^{- \mu \sqrt{K} t^{2/3} \sigma}$ is decreasing in $\sigma$. We first condition the right hand side above on $\fF_{\tau_{m-1}}$. Note that by Proposition~\ref{prop:hitting_time}, we have
\begin{equation*}
    \begin{split}
    \EE \Big( \phi(\widehat{\sigma}_m) \, \1_{\{\widehat{\sigma}_m < t - \tau_{m-1}\}} | \fF_{\tau_{m-1}} \Big) &= \EE^{W_{\tau_{m-1}}} \Big( \phi(\widehat{\sigma}_m) \, \1_{\{\widehat{\sigma}_m < t - \tau_{m-1}\}} \Big)\\
    &\leqslant \int_{0}^{t - \tau_{m-1}} q_m^\alpha (u) \phi(u) du\\
    &\leqslant \int_{0}^{t - \tau_{m-2} - \widehat{\sigma}_{m-1}} q_m^\alpha (u) \phi(u) du\;,
    \end{split}
\end{equation*}
where we used $\widehat{\sigma}_{m-1} \leqslant \sigma_{m-1} - \tau_{m-2} \leqslant \tau_{m-1} - \tau_{m-2}$ in the last inequality. Hence, the expectation part of right hand side of \eqref{eq:KeyEstPf2} is bounded by
\begin{equation*}
    \EE \bigg[ \bigg( \int_{0}^{t-\tau_{m-2} - \widehat{\sigma}_{m-1}} q_m^\alpha(u) \phi(u) du \bigg) \cdot \prod_{k=1}^{m-1} \Big( \phi(\widehat{\sigma}_k) \, \1_{\{ \widehat{\sigma}_k < t - \tau_{k-1} \}} \Big) \bigg]\;.
\end{equation*}
Note that both the integral and $\phi(\widehat{\sigma}_{m-1})$ are decreasing in the variable $\widehat{\sigma}_{m-1}$, and hence so is their product. This allows us to further condition on $\fF_{\tau_{m-2}}$ and apply Proposition~\ref{prop:hitting_time} to get another decreasing function in $\widehat{\sigma}_{m-3}$. Proceeding in this way (by successively conditioning on $\fF_{\tau_k}$ for $k=m-1, \dots, 1$), we finally get
\begin{equation*} 
    \EE \bigg[ \prod_{k=1}^{m} \Big( \phi(\widehat{\sigma}_k) \, \1_{\{ \widehat{\sigma}_k < t - \tau_{k-1} \}} \Big) \bigg] \leqslant \int_{\sS_m(t)} \Big( \prod_{k=1}^{m} q_k^{\alpha}(u_k) \Big) \cdot \exp \Big( - \mu \sqrt{K} t^{2/3} \sum_{k=1}^{m} u_k \Big) d {\bf u}\;,
\end{equation*}
where
\begin{equation} \label{eq:range_integration}
    \sS_m(t) = \big\{ \mathbf{u} = (u_1, \dots, u_m): 0 < u_1 + \cdots + u_m < t  \big\}\;.
\end{equation}
Using the expression of $q_k^\alpha$ and plugging it back to \eqref{eq:KeyEstPf2}, we get
\begin{equation} \label{eq:IDef}
    \begin{split}
    \EE \Big[ e^{\int_{0}^{t} \xi(W_s) ds} \, &\1_{\{\tT_{\delta,\eta}^t = \vec{\fS}\}} \Big] \leqslant e^{\delta t^{5/3}} \int_{\sS_m(t)} \bigg[ \exp \bigg( \mu \sqrt{K} t^{2/3} \Big( t - \sum_{k=1}^{m} u_k \Big)  - \sum_{k=1}^{m} \frac{\alpha^4 D_k^2}{4 u_k} \bigg)\\
    &\cdot \prod_{k=1}^{m} \bigg( \frac{\alpha D_k} {2\sqrt{\pi}} \cdot u_k^{-3/2} \cdot \exp \Big(-\frac{\alpha^2 (1-\alpha^2) D_k^2}{4 u_k} \Big) \bigg)  \bigg] \, d {\bf u}\;,
    \end{split}
\end{equation}
where we have further split $\alpha^2 = \alpha^4 + \alpha^2 (1-\alpha^2)$.

\subsubsection{Taking out the main factor}

We now analyse the term on the first line of the integrand on the right hand side of \eqref{eq:IDef} above. This is the source of the $t^{5/3}$ power and the constant $L^*$. We first give two preliminary ingredients. 

\begin{lem} \label{lem:ElemIneq}
We have
\begin{equation} \label{eq:ElemIneq}
\sum_{k=1}^{m}\frac{D_k^{2}}{u_k} \geqslant \frac{(D_1 + \dots + D_m)^{2}}{u_1 + \dots + u_m}
\end{equation}
for every integer $m \in \NN$ and all positive real numbers $D_1, \dots, D_m$ and $u_1, \dots, u_m$. 
\end{lem}
\begin{proof}
We first consider $m=2$. Let $a = \frac{D_1}{D_1 + D_2}$ and $b = \frac{u_1}{u_1 + u_2}$. Then we have
\begin{equation*}
    \frac{D_1^2}{u_1} + \frac{D_2^2}{u_2} \geqslant \frac{(D_1 + D_2)^2}{u_1 + u_2} \iff \frac{a^2}{b}+\frac{(1-a)^2}{(1-b)} \geqslant 1 \iff (a-b)^2 \geqslant 0\;.
\end{equation*}
The case for general $m$ then follows by induction. 
\end{proof}

\begin{lem} \label{lem:Trian}
We have
\begin{equation*}
    \sum_{k=1}^{m} D_k \geqslant (K-\delta^2) \cdot t^{4/3}\;.
\end{equation*}
\end{lem}
\begin{proof}
    Recall $\tT_{\delta,\eta}^t = \vec{\fS}$ with $|\vec{\fS}| = m$. Let $m^* \leqslant m$ satisfying \eqref{eq:furthest_cluster} and $\rR(\tT_{\delta,eta}^t) = \rR(\vec{\fS})$ be the reduced cluster constructed as in Definition~\ref{def:reduced_cluster} based on the above $m^*$. Suppose $\rR(\tT_{\delta,\eta}^t)$ has the form
    \begin{equation*}
        \rR(\tT_{\delta,\eta}^t) = (\fS_{j_1}, \, \dots, \, \fS_{j_{\bar{m}}})\;.
    \end{equation*}
    By construction, we have $j_{\bar{m}} = m^*$, and $\vec{\fS}_{j_{\ell+1}} = \vec{\fS}_{j_\ell + 1}$ for every $\ell$. The latter in particular implies
    \begin{equation} \label{eq:distance_reduced_cluster}
        D_{j_{\ell+1}} = {\rm dist} \, (\vec{\fS}_{j_\ell}, \vec{\fS}_{j_{\ell+1}}) - \frac{\eta \, t^{4/3}}{2}\;.
    \end{equation}
    Now, for each $\ell \leqslant \bar{m}$, take $x_\ell \in \overline{\fC_{j_\ell}}$ with $x_0 = o$. By triangle inequality and Assertion 1 in Lemma~\ref{lem:ClusterProperty} on the diameter of the clusters in $\Clus_{\delta,\eta}^t$, we have
    \begin{equation*}
        d(x_\ell, x_{\ell+1}) \leqslant {\rm dist} \big( \vec{\fS}_{j_\ell}, \\vec{\fS}_{j_{\ell+1}} \big) + {\rm diam} (\vec{\fS}_{j_\ell}) + {\rm diam} (\vec{\fS}_{j_{\ell+1}}) \leqslant D_{j_{\ell+1}} + \frac{\eta \, t^{4/3}}{2} + 8 L_\delta \eta t^{4/3}\;.
    \end{equation*}
    Summing $\ell$ from $0$ to $\bar{m}-1$, and using the bound for $\bar{m}$ in Lemma~\ref{lem:reduced_cluster_length} and that all the $j_\ell$'s different for different $\ell$, we get
    \begin{equation*}
        \begin{split}
        K t^{4/3} \leqslant d(o, x_{\bar{m}}) + L_\delta \, \eta \, t^{4/3} &\leqslant \sum_{\ell=0}^{\bar{m}-1} d(x_\ell, x_{\ell+1}) +  L_\delta \, \eta \, t^{4/3}\\
        &\leqslant \sum_{\ell=1}^{\bar{m}} D_{j_\ell} + 10 \bar{m} \, L_\delta \, \eta \, t^{4/3} \leqslant \sum_{k=1}^{m} D_{k} + C \, \delta^3 \, t^{4/3}\;.
        \end{split}
    \end{equation*}
    The constant $C$ can be removed by reducing $\delta^3$ to $\delta^2$ if $\delta$ is sufficiently small. This completes the proof. 
\end{proof}

We are now ready to state the main bound. 

\begin{prop} \label{prop:L_star}
    We have
    \begin{equation} \label{eq:upper_bound_L_star}
        \mu \sqrt{K} t^{2/3} \Big( t - \sum_{k=1}^{m} u_k \Big) - \sum_{k=1}^{m} \frac{\alpha^4 D_k^2}{4 u_k} \leqslant \big( L^* + \mu \sqrt{K} - \alpha \mu_0 \sqrt{K-\delta^2} \big) \cdot t^{5/3}
    \end{equation}
    uniformly over ${\bf u} = (u_1, \dots, u_m) \in \sS_m(t)$, $\theta \in (0,1)$ and all sufficiently large $t$, where $L^*$ is the constant defined in \eqref{eq:OptimalExp}, and $K = K_t$ is given in \eqref{eq:K*Def}. 
\end{prop}
\begin{proof}
    Combining Lemmas~\ref{lem:ElemIneq} and~\ref{lem:Trian}, we have
    \begin{equation*}
        \sum_{k=1}^{m} \frac{D_k^2}{u_k} \geqslant \frac{(D_1 + \cdots D_m)^2}{u_1 + \cdots u_m} \geqslant \frac{(K-\delta^2)^2 \; t^{8/3}}{u_1 + \cdots + u_m}\;.
    \end{equation*}
    Writing $\sum_{k=1}^{m} u_k = \theta t$ for some $\theta \in [0,1]$ and recalling the definition of $L^*$ in \eqref{eq:OptimalExp}, we can bound the left hand side of \eqref{eq:upper_bound_L_star} by
    \begin{equation*}
        \begin{split}
        \mu \sqrt{K} t^{2/3} \Big( t - \sum_{k=1}^{m} u_k \Big) - \sum_{k=1}^{m} \frac{\alpha^4 D_k^2}{4 u_k} &= \Big( \mu \sqrt{K} (1-\theta) - \frac{\big( \alpha^2 (K-\delta^2) \big)^2}{4 \theta} \Big) \, t^{5/3}\\
        &\leqslant \Big( L^* + (1-\theta) \big( \mu \sqrt{K} -  \mu_0 \sqrt{\alpha^2 (K-\delta^2)} \big) \Big) \, t^{5/3}\;.
        \end{split}
    \end{equation*}
    This completes the proof of Proposition~\ref{prop:L_star}. 
\end{proof}

\subsubsection{Proof of Proposition~\ref{prop:KeyEst}}

\begin{lem} \label{lem:remainder_integral}
    There exists $c>0$ independent of $m$, $\eta$ and $t$ such that
    \begin{equation*}
        \int_{\sS_m(t)} \prod_{k=1}^{m} \bigg( \frac{\alpha D_k} {2\sqrt{\pi}} \cdot u_k^{-3/2} \cdot \exp \Big(-\frac{\alpha^2 (1-\alpha^2) D_k^2}{4 u_k} \Big) \bigg) \; d {\bf u} \leqslant e^{- c m \eta^2 t^{5/3}}
    \end{equation*}
    for all $m \geqslant 1$, $\eta \in (0,1)$ and $t$ sufficiently large such that $\eta^2 t^{5/3} \gg 1$. 
\end{lem}
\begin{proof}
    Since $\eta t^{4/3}/2 \leqslant D_k \leqslant 2 K_0 t^{4/3}$ for every $k$, we have the bound (enlarging the domain of integration to $[0,t]^m$)
    \begin{equation*}
        \begin{split}
        &\phantom{111}\int_{\sS_m(t)} \prod_{k=1}^{m} \bigg( \frac{\alpha D_k} {2\sqrt{\pi}} \cdot u_k^{-3/2} \cdot \exp \Big(-\frac{\alpha^2 (1-\alpha^2) D_k^2}{4 u_k} \Big) \bigg) \; d {\bf u}\\
        &\leqslant \bigg( \frac{\alpha K_0}{\sqrt{\pi}} \cdot t^{4/3} \int_{0}^{t} u^{-3/2} \, \exp \Big( -\frac{\alpha^2 (1-\alpha^2) \eta^2 \, t^{8/3}}{16 u} \Big) du \bigg)^m\;.
        \end{split}
    \end{equation*}
    Making the change of variable $v \triangleq u / (\eta^2 \, t^{8/3})$, we see the right hand side above equals
    \begin{equation*}
        C^m \, \bigg( \int_{0}^{1/(\eta^2 t^{5/3})} v^{-3/2} e^{-\frac{c}{v}} dv \bigg)^m
    \end{equation*}
    for some $c, C>0$ independent of $\eta$ and $t$. The claim then follows. 
\end{proof}

\begin{proof} [Proof of Proposition~\ref{prop:KeyEst}]
    The conclusion of Proposition~\ref{prop:KeyEst} follows immediately by combining Proposition~\ref{prop:L_star} and Lemma~\ref{lem:remainder_integral} and plugging them back into the expression \eqref{eq:IDef}. 
\end{proof}

\subsection{Completing the proof of Theorem \ref{thm:MainUp}}

Now we are in a position to complete the proof of the main upper bound. 

\begin{proof}[Proof of Theorem \ref{thm:MainUp}]

Combining \eqref{eq:full_localisation} and Proposition~\ref{prop:KeyEst}, we get
\begin{equation} \label{eq:upper_sum}
    u(t,o) \leqslant e^{- c' \,  t^{5/3}} + e^{(L^* + \delta + \mu \sqrt{K_0} - \alpha \mu_0 \sqrt{K_0 - \delta^3}) t^{5/3}} \, \sum_{\vec{\fS} \in \gG_{\delta,\eta,\eta'}^t} e^{- c \, |\vec{\fS}| \, \eta^2 \, t^{5/3}}\;.
\end{equation}
Note that clusters in $\Clus_{\delta,\eta}^t$ are at least $\eta t^{4/3}$ away from each other, and all are contained in $Q_{K_0 t^{4/3}}$. Hence, the total number of different clusters in $\Clus_{\delta,\eta}^t$ is bounded by $C e^{(d-1) K_0 t^{4/3}}$ for some $C$ depending on $\delta$ and $d$. As a consequence, the total number of $\vec{\fS} \in \gG_{\delta,\eta,\eta'}^t$ with $|\vec{\fS}| = m$ is bounded by $e^{C m t^{4/3}}$ for some $C>0$ independent of $\eta$ and $t$ as long as $\eta \, t^{4/3} \geqslant 1$. Plugging this back to the sum \eqref{eq:upper_sum}, we deduce that for every $\delta > 0$ sufficiently small, $\mu > \mu_0$ and $\alpha \in (0,1)$, we have
\begin{equation*}
    \underset{t\rightarrow\infty}{\overline{\lim}}\frac{1}{t^{5/3}}\log u(t,o) \leqslant L^* + \delta + \mu \sqrt{K_0} - \alpha \mu_0 \sqrt{K_0 - \delta^2}\;.
\end{equation*}
Since $\delta > 0$ is arbitrarily small, $\mu > \mu_0$ and $\alpha < 1$ can also be arbitrarily close to $\mu_0$ and $1$, the proof of Theorem~\ref{thm:MainUp} is thus complete. 
\end{proof}

\section*{Acknowledgement}

XG is supported by ARC grant DE210101352. WX acknowledges the support from the Ministry of Science and Technology via the National Key R\&D Program of China (no.2023YFA1010102) and National Science Foundation China via the Standard Project Grant (no.12171008) and the Key Project Grant (no.12595280 and 12595281). 

The majority of the work was done when WX was at Beijing International Center for Mathematical Research (BICMR) at Peking University, which provides a stimulating environment for mathematical research. 

The authors thank Stephen Muirhead for valuable discussions and suggestions which lead to various improvements of the current work.

\begin{appendices}

\section{Proof of Lemma~\ref{lem:LDP}}
\label{sec:LDP}

It is convenient to re-parametrise everything on the time interval $[0,1]$. For $\omega: [0,1] \rightarrow \HH^d$, define its energy $\eE$ by
\begin{equation*}
    \eE (\omega) \triangleq
    \begin{cases}
    \int_{0}^{1}|\dot{\omega}_u|^2 du, & \text{if }|\dot{\omega}|\in L^{2}([0,1]);\\
    +\infty, & \text{otherwise}.
    \end{cases}
\end{equation*}
We recall the following classical large deviations property for Brownian motion(cf. \cite{FW98}).

\begin{lem} \label{lem:LDPBM}
Let $F$ be a closed subset of the path space 
\begin{equation*}
    \Omega \triangleq \{\omega:[0,1]\rightarrow \HH^d \ |\ \omega\text{ continuous,}\ \omega_{0}=o\}
\end{equation*}
with respect to the uniform topology. Then we have 
\begin{equation*}
    \underset{s\rightarrow0^{+}}{\overline{\lim}} s \log \PP_o \big( W^s \in F \big) \leqslant -\frac{1}{4} \inf_{\omega \in F} \eE(\omega)\;,
\end{equation*}
where $W^s$ is the time-parametrised Brownian motion given by $W^s_u \triangleq W_{su}$ with $u \in [0,1]$. 
\end{lem}

Recall from Lemma~\ref{lem:LDPBM} that $\alpha: [0,s] \rightarrow \HH^d$ is the geodesic of fixed length $L$ starting from $o$. Let $\alpha^s: [0,1] \rightarrow \HH^d$ be defined by $\alpha^s_u \triangleq \alpha_{su}$. Then $\alpha^s$ is independent of $s$. We choose $F$ to be the closed subset
\begin{equation} \label{eq:closed_subset}
    F = \Big\{ \omega \in \Omega: \sup_{u \in [0,1]} d(\omega_u, \alpha^s_u) \geqslant \eps\,, \; \; d(\omega_1, \alpha^s_1) \leqslant \zeta \Big\}\;.
\end{equation}
We need to demonstrate an effective lower bound for the energy functional $\eE(\omega)$ on $F$. This is given by the lemma below whose proof relies on the assumption of non-positive curvature.

\begin{lem} \label{lem:EneLow}
We have
\begin{equation} \label{eq:EneLow}
\eE(\omega) \geqslant L^2 + \frac{\eps^{2}}{32} - 2 \zeta \, L
\end{equation}
for all $\omega \in F$ defined in \eqref{eq:closed_subset}, where $L$ is the length of $\alpha^s$. 
\end{lem}
\begin{proof}
We first recall the following convexity property of geodesics in the
hyperbolic space (\cite[Theorem 4.8.2]{Jos05}). Let $\alpha,\beta:[0,T] \rightarrow \HH^d$ be geodesics. Then we have 
\begin{equation} \label{eq:GeoConv}
\sup_{u \in [0,T]} d(\alpha_u, \, \beta_u) \leqslant \max\big\{ d(\alpha_0, \, \beta_0), \; d(\alpha_T, \, \beta_T) \big\}\;.
\end{equation}
Now fix $\omega\in F$. Let $\beta:[0,1] \rightarrow \mathbb{H}^{d}$
be the geodesic from $o$ to $\omega_1$. Using triangle inequality and then applying \eqref{eq:GeoConv} to the geodesics $\alpha^s$ and $\beta$, we have
\begin{equation*}
    \sup_{u \in [0,1]} d(\omega_u,\beta_u) \geqslant \sup_{u \in [0,1]} d(\omega_u, \alpha_u^s) - \sup_{u \in [0,1]} d(\alpha^s_u, \beta_u) \geqslant \frac{\eps}{3} - \zeta(t) > \frac{\eps}{4}\;,
\end{equation*}
where in the last step we also used the assumption that $\zeta \ll \eps$. 

Fix any $\tau \in (0,1)$ such that $d(\omega_\tau, \beta_\tau) > \frac{\eps}{4}$. Let $\bar{\omega}$ be the piecewise geodesic
\begin{equation*}
    \bar{\omega}_u \triangleq
    \begin{cases}
\text{geodesic from \ensuremath{o} to \ensuremath{\omega_\tau},} & r \in [0,\tau];\\
\text{geodesic from }\omega_{v}\ \text{to }\omega_{1}, & r \in [\tau,1].
\end{cases}
\end{equation*}
Since geodesics minimise energy, we have 
\begin{equation*}
    \eE (\omega)=\int_{0}^{\tau} |\dot{\omega}_u|^{2} du + \int_{\tau}^{1} |\dot{\omega}_u|^{2} du \geqslant \int_{0}^{\tau} |\dot{\bar{\omega}}_u|^2 du + \int_{\tau}^{1} |\dot{\bar{\omega}}_u|^2 du = \eE (\bar{\omega})\;.
\end{equation*}
It remains bound $\eE (\bar{\omega})$ from below. We will use the second variation of energy formula to achieve this. 

Let $\varphi: [0,1] \rightarrow \HH^d$ be the geodesic parametrised at uniform speed with $\varphi_0 = \beta_\tau$ and $\varphi_1 = \omega_\tau = \bar{\omega}_\tau$. For each $\theta \in [0,1]$, define
\begin{equation*}
    h^\theta_u \triangleq
    \begin{cases}
\text{geodesic from \ensuremath{o} to \ensuremath{\varphi_\theta},} & u \in [0, \tau]\;;\\
\text{geodesic from } \varphi_{\theta}\ \text{to } \omega_{1}=\bar{\omega}_{1}, & u \in [\tau,1]\;.
\end{cases}
\end{equation*}
Then we have $h^0 = \beta$ and $h^1 = \bar{\omega}$. 
We set $\dot{h} \triangleq \partial_u h$ and $\partial_\theta h$ to be the tangential and variational vector fields respectively. With an abuse of notation, we write
\begin{equation*}
    \eE(\theta) \triangleq \int_{0}^{1} |\dot{h}^\theta_u|^ du
\end{equation*}
be the energy of $h^\theta$, with $\eE(0) = \eE(\beta)$ and $\eE(1) = \eE(\bar{\omega})$. 

According to the second variation of energy formula (\cite[Theorem 4.1.1]{Jos05}) and negativity of curvature, we have
\begin{equation} \label{eq:EneLow1}
    \eE''(\theta) = \int_{0}^{1} |(\nabla_{\partial_{u}} \partial_\theta h^\theta)(u)|^{2} du - \int_{0}^{1} \langle R(\dot{h}, \partial_\theta h) \partial_\theta h, \dot{h} \rangle du \geqslant \int_{0}^{1} |(\nabla_{\partial_u} \partial_\theta h)(u)|^2 du\;.
\end{equation}
Here $R(\cdot,\cdot)\cdot$ is the Riemannian curvature tensor, and the inequality comes from the negativity of the curvature of $\HH^d$. 

To bound the integral on the right hand side from below, let $\{e_{1},\cdots,e_{d}\}$ be an orthonormal basis of $T_o \HH^d$ and parallel translate it along the geodesics $(h^\theta_u)_{u \in [0,\tau]}$ to obtain an orthonormal frame field $\{E_{1},\cdots,E_{d}\}$. We write
\begin{equation*}
    \partial_\theta h^\theta_u = \sum_{j=1}^{d} C^\theta_j (u) \, E_j\;, \quad u \in [0,\tau].
\end{equation*}
Note that $\partial_\theta h^\theta_0 = 0$ for all $\theta$. This implies
\begin{equation*}
    |C^\theta_j(u)|^2 = \Big( \int_{0}^{u} \partial_{r} C^\theta_j (r) dr \Big)^{2} \leqslant \tau \int_{0}^{\tau} \Big(\partial_{r} C^\theta_j (r) \Big)^{2} d r\;.
\end{equation*}
Thus, we have
\begin{equation} \label{eq:EneLow2}
|\partial_\theta h^\theta_u|^{2} = \sum_{j=1}^{d} |C^\theta_j (u)|^{2} \leqslant \tau \sum_{j=1}^{d} \int_{0}^{\tau} \big| \partial_{r} C^\theta_j(r) \big|^{2} d r \leqslant \int_{0}^{1} \big|(\nabla_{\partial_{r}} \partial_\theta h^\theta)(r) \big|^2 d r
\end{equation}
for all $u \in[0,\tau]$. The same inequality also holds for $r \in [\tau,1]$ (by the same argument). As a result, one concludes from \eqref{eq:EneLow1} and \eqref{eq:EneLow2} that 
\begin{equation*}
  \eE''(\theta) \geqslant \sup_{u \in[0,1]} |\partial_\theta h^\theta_u|^{2} \geqslant \frac{\eps^{2}}{16}\;,
\end{equation*}
where the last inequality follows from the fact that $\big| \partial_\theta h^\theta_\tau |_{\theta = 0} \big| = d(\bar{\omega}_\tau, \beta_\tau) > \frac{\eps}{4}$. Since $\eE'(0)=0$, it then follows that 
\begin{equation*}
    \eE(\bar{\omega}) = \eE(1) = \eE(0) + \int_{0}^{1} (1-\theta) \eE''(\theta) d \theta \geqslant d^2(o,\omega_{1}) + \frac{\eps^2}{32}\;.
\end{equation*}
A further application of the triangle inequality yields
\begin{equation*}
    d^2 (o,\omega_{1}) \geqslant \big( d(o,\alpha^s_1) -d(\alpha^s_1, \omega_{1}) )^{2} \geqslant L^2 - 2 \zeta L\;.
\end{equation*}
We then obtain the desired inequality \eqref{eq:EneLow} and completes the proof of the lemma.
\end{proof}

\begin{rem}
Lemma \ref{lem:EneLow} may fail in positively curved manifolds. For
instance, all great semicircles connecting the north and south poles
on the sphere are geodesics and they have the same energy. The lemma
holds on general Riemannian manifolds when the geodesic does not contain conjugate points and the deviation is small.
\end{rem}

\end{appendices}

\end{document}